\documentclass[11pt]{amsart}
\usepackage[all]{xy}
\usepackage{amsmath}
\usepackage{amsfonts}
\usepackage{amssymb}
\usepackage{amscd}
\usepackage{amsthm}
\usepackage{latexsym}
\usepackage{amsbsy}
\usepackage{xcolor,enumerate}
\usepackage[normalem]{ulem}
\usepackage{epigraph}
\usepackage{tikz}

\usetikzlibrary{arrows,calc,through,backgrounds,matrix,decorations.pathmorphing}
\tikzset{node distance=2cm, auto}

\def\0D{\Delta^{(0)}}
\def\1D{\Delta^{(1)}}

\newcommand{\Zc}{\mathcal{Z}}

\newcommand{\hmod}{{_H\mathcal{M}}}
\newcommand{\hmodalg}{{_\mathcal{H} \mathcal{M}}}

\newtheorem{theorem}{Theorem}[section]
\newtheorem{proposition}[theorem]{Proposition}
\newtheorem{lemma}[theorem]{Lemma}
\newtheorem{corollary}[theorem]{Corollary}

\newtheorem{definition}[theorem]{Definition}
\theoremstyle{remark}
\newtheorem{remark}[theorem]{Remark}

\def\build#1_#2^#3{\mathrel{\mathop{\kern 0pt#1}\limits_{#2}^{#3}}}

\newcommand{\vect}{\text{Vec}}
\newcommand{\id}{1}

\newcommand{\la}{\triangleright}
\newcommand{\ra}{\triangleleft}

\numberwithin{equation}{section}

\def\i{\iota}

\def\ot{\otimes}
\def\part{\partial}

\def\ot{\otimes}

\def\Hom{\mathop{\rm Hom}\nolimits}

\def\build#1_#2^#3{\mathrel{
\mathop{\kern 0pt#1}\limits_{#2}^{#3}}}

\numberwithin{equation}{section}
\newcommand{\comment}[1]{\relax}

\begin{document}
\title[cyclic cohomology of quasi-Hopf algebras and Hopf algebroids]{A categorical approach to cyclic cohomology of quasi-Hopf algebras and Hopf algebroids.}

\author[I. Kobyzev]{Ivan Kobyzev}

\address{Department of Pure Mathematics, University of
	Waterloo, Waterloo, ON N2L 3G1, Canada}

\email{ikobyzev@uwaterloo.ca}

\author[I. Shapiro]{Ilya Shapiro}

\address{Department of Mathematics and Statistics,
	University of Windsor, 401 Sunset Avenue, Windsor, Ontario N9B 3P4, Canada}

\email{ishapiro@uwindsor.ca}

\maketitle
\begin{abstract}
We apply categorical machinery to the problem of defining cyclic cohomology with coefficients in two particular cases, namely quasi-Hopf algebras and Hopf algebroids.  In the case of the former, no definition was thus far available in the literature, and while a definition exists for the latter, we feel that our approach demystifies the seemingly arbitrary formulas present there.  This paper emphasizes the importance of working with a biclosed monoidal category in order to obtain natural coefficients for a cyclic theory that are analogous to the stable anti-Yetter-Drinfeld contramodules for Hopf algebras.

\end{abstract}

\medskip

{\it 2010 Mathematics Subject Classification.} Monoidal category (18D10), abelian and additive category (18E05), cyclic homology (19D55), Hopf algebras	 (16T05).

\epigraph{\textit{Everything that's extreme is difficult. The middle parts are done more easily. The very centre requires no effort at all. The centre is equal to equilibrium. There's no fight in it.} }{Daniil Kharms, \textit{Blue notebook}}

\section{Introduction}
After its introduction by Tsygan and Connes independently in the 1980s, cyclic homology has been the subject of extensive research.  An equivariant version for Hopf algebras begun with Connes-Moscovici \cite{CM98, CM99} and culminated with the definition of stable anti-Yetter-Drinfeld module coefficients by Hajac-Khalkhali-Rangipour-Sommerhauser \cite{hkrs1, hkrs2} and independently by Jara-Stefan \cite{js}, and stable anti-Yetter-Drinfeld contramodule coefficients by Brzezinski \cite{contra}.

Beyond Hopf algebras one has a number of Hopf-like notions, relaxing the various axioms of a Hopf algebra.  One is naturally interested in extending an equivariant version of cyclic homology, together with the relevant coefficients, to these more general settings. The problem turns out to be highly non-trivial as the definitions rely almost entirely on the correct notion of coefficients; these are defined by complicated formulas that are surprisingly difficult to generalize.

An approach to this problem is afforded by \cite{hks} where it is shown that the monoidal-categorical point of view provides a way to extract the necessary formulas from a conceptual model of the situation.  It is demonstrated that indeed one obtains the usual definitions in the Hopf algebra case.   Here we continue the application of the machinery developed there.  Another key observation for us is the importance of the existence of internal homomorphisms in our monoidal category; more precisely these are right adjoints to the tensor product.  A category possessing internal homs is said to be biclosed which refers to the existence of both left and right internal homs.  This property turns out to be the sole ingredient required in order to define an analogue of stable anti-Yetter-Drinfeld contramodule coefficients in the more general Hopf-like setting.

In this paper we deal with two particular cases of Hopf-like objects.  However, before we begin with these applications we examine the setting of Hopf algebras, namely that of stable anti-Yetter-Drinfeld contramodule coefficients as in \cite{contra}.  Our approach yields the same definitions, but the transition from the categorical perspective to formulas is very straightforward and serves both as motivation and contrast to the main two cases of this paper considered subsequently.

We begin with addressing the issue of defining cyclic cohomology for quasi-Hopf algebras, these were originally introduced by Drinfeld in \cite{Drinfeld}.  Roughly speaking, we obtain quasi-Hopf algebras by weakening the co-associativity condition in the definition of a Hopf algebra to that of co-associativity up to a prescribed isomorphism. An important classical fact is that the category of representations of a quasi-Hopf algebra is still a monoidal category, and what is crucial for us is that this monoidal category is biclosed.  From quite general considerations one can then provide a conceptual description of coefficients, followed by the unraveling of the definitions to obtain explicit and immensely unpleasant formulas.  The complexity of the resulting expressions explains the absence of any definition of cyclic homology for quasi-Hopf algebras thus far.

The second case that we consider is that of Hopf algebroids.  Various related concepts purporting to describe these objects exist and we choose to work with the definitions in \cite{kow} which follow those of \cite{Bohm}.  Again, roughly, the generalization from Hopf algebras consists of replacing the ground field $k$, over which everything is tensored, by a noncommutative ring $R$.  Thus the objects underlying representations of a Hopf algebroid consist not of vector spaces over $k$, but of bimodules over $R$.  It seems that Hopf algebroids provide an example of a construct that is the most noncommutative in noncommutative geometry.

We point out that a definition of cyclic cohomology with an analogue of stable anti-Yetter-Drinfeld contramodule coefficients has been given by  Kowalzig in \cite{kow_contra}.  Our goal here is to explain the formulas that appear there as an inevitable consequence of the general machinery.

The paper is organized as follows.  In Section \ref{s:General} we discuss some generalities as in \cite{hks, s} concerning a conceptual definition of contramodule coefficients for a suitable monoidal category.  A new phenomenon that arises in this paper is the difficulty of proving that a certain natural map is an isomorphism.  It requires that we deal with weak centers as opposed to (strong) centers; this particular issue does not occur with Hopf algebras.  Fortunately we are able to sidestep the problem by showing that stability, a condition that is required in any case for a Hopf-type theory, guarantees that our objects of interest are indeed in the strong center.

Section \ref{s:Hopf} deals with the motivational case of contramodule coefficients for Hopf algebras.  In Section \ref{s:qHopf} we recall the definition of quasi-Hopf algebras and give a prove that their categories of modules are biclosed (this is a fact known to the experts in the field \cite{BPO_closed, majid_closed, SS_closed}). In Section \ref{s:qHopf-ayd-contra} we unravel the categorical definitions of Section \ref{s:General} to obtain definitions of anti-Yetter-Drinfeld contramodules for quasi-Hopf algebras.  We remark that unlike the Hopf algebra case, there are two different descriptions via formulas of the conceptually defined category of coefficients.    Stability for type I contramodules is explicitly written down as well.

The Sections \ref{s:Hopfroids} and \ref{s:Hopfroid-ayd-contra} address the definitions of Hopf algebroids, the biclosed property of their categories of modules (we provide a direct proof, but the fact can also be obtained from \cite{sch_closed}) and finally the translation into formulas of the definition of anti-Yetter-Drinfeld contramodules and their stability.

We remark that we handled the case of anti-Yetter-Drinfeld \emph{module} analogues in our follow-up paper \cite{KS}.  We did not want to increase the length of the exposition any further and modules, despite being more familiar, are actually technically more difficult and less natural/suitable from the categorical perspective on the subject of coefficients.

\bigskip

\textbf{Acknowledgments}: The authors wish to thank Masoud Khalkhali for stimulating questions and discussions.  Furthermore, we are grateful to the referee for the constructive comments and useful references. The research of the second author was supported in part by the NSERC Discovery Grant number 406709.

\section{Generalities}\label{s:General}

In this section we will establish the general formalism for cyclic cohomology with anti-Yetter-Drinfeld contramodule coefficients.  Recall from \cite{hks} that the main ingredient in constructing Hopf-cyclic cohomology is a symmetric 2-contratrace. We will briefly recall the necessary definitions.

Let $\mathcal{M}$ be a monoidal category. Consider  an $\mathcal{M}$-bimodule category $\mathcal{M}^*=\text{Fun}(\mathcal{M}^{op}, \vect )^{op}$, where $\vect$ denotes $k$-vector spaces,  with actions defined by:
$$ M \la F(-) := F(- \ot M ), \quad F(-) \ra M := F(M \ot -).  $$
Any element in $\Zc_{\mathcal{M}}(\mathcal{M}^* ) $ is called a \emph{2-contratrace}. In particular it requires the existence of  natural isomorphisms
$$\i_M: F(- \ot M) \rightarrow F(M \ot - ), $$ and a 2-contratrace is called symmetric, if the composition $F(M\ot M')\to F(M'\ot M)\to F(M\ot M')$ is $Id$ for all $M,M'\in \mathcal{M}$.

By \cite[Proposition 3.9]{hks} a unital associative algebra object $A$ in $\mathcal{M}$ and a symmetric 2-contratrace $F$ provides us with a cocyclic object $C^\bullet = F(A^{\bullet+1})$; its cyclic cohomology recovers  exactly the classical Hopf-cyclic cohomology of an algebra $A$ with coefficients in a stable anti-Yetter-Drinfeld contramodule for the case $\mathcal{M}=\hmod$, the category of modules over a Hopf algebra $H$.  More precisely, $F(-)=\Hom_H(-,M)$ where $M$ is a stable anti-Yetter-Drinfeld contramodule.  This is explained in Section \ref{s:Hopf}.

A crucial element for our constructions is a biclosed monoidal category $\mathcal{M} $. The property of being biclosed implies in particular the existence of the following adjunctions
for $M, V, W\in \mathcal{M}$:
$$ \ \Hom_\mathcal{M} (W\otimes V, M) \simeq \Hom_\mathcal{M} (W, \Hom^l(V,M)),$$

and
$$\Hom_\mathcal{M} (V\otimes W, M) \simeq \Hom_\mathcal{M} (W, \Hom^r(V,M)),$$
where $\Hom^l(V,M) $ and $\Hom^r(V,M)$ are left and right internal homomorphisms respectively.

As in \cite{s}, we can introduce the contragradient $\mathcal{M}$-bimodule category $\mathcal{M}^{op}$. Specifically, for $M\in  \mathcal{M}^{op}$ and $V\in  \mathcal{M}$, the actions are given by:
\begin{equation}
M\ra V := \Hom^r(V, M), \quad ~~  \quad ~~ V\la M :=\Hom^l(V, M)  .
\end{equation}

The natural object to consider is the center of this bimodule contragradient category: $\Zc_{\mathcal{M}}(\mathcal{M}^{op})$ (see \cite{GNN} for the definition of the center of a bimodule category). However, in some situations (e.g. quasi-Hopf algebras or Hopf algebroids) it becomes too restrictive. If in the definition of a (strong) center element $M\in\mathcal{M}^{op}$ we relax the condition that the maps $\tau:M'\la M\to M\ra M'$ are isomorphisms, we get a \emph{weak} center. More formally:

\begin{definition}
	The weak center $ \text{w-}\Zc_{\mathcal{M}}(\mathcal{N})$ of a
	$\mathcal{M}$-bimodule category $\mathcal{N}$ consists of objects that are pairs $(
	N, \tau_{N,-})$, where $N \in \mathcal{N}$ 	and $\tau_{N,-}$	is a family of natural
	morphisms such that for $V \in \mathcal{M}$ we have: $\tau_{N,V}: V \la N \to N \ra V $, satisfying the hexagon axiom as in the usual definition of center.

	A morphism $(N, \tau_{N,-}) \rightarrow (N', \tau_{N',-})$ is a  morphisms $t: N\rightarrow N'$ in  $\mathcal{N}$ satisfying ${\tau_{N',V}\circ (id_V\triangleright t)} =  (t\triangleleft id_V)  \circ \tau_{N,V}$ for all $V \in \mathcal{M}$.
\end{definition}

\begin{remark}
	One can give a reformulation of this definition in a categorical language. Consider an $\mathcal{M}$-bimodule category $\mathcal{N}$ as a weak 2-category $\mathcal{B}$ with two objects $L$ and $R$, where hom-categories are: $\mathcal{B}(L;L) = \mathcal{B}(R;R) = \mathcal{M}$, $\mathcal{B}(L;R) = \mathcal{N}$ and
	$\mathcal{B}(R;L) =\emptyset$. The composition of 1-cells given by the monoidal product of $\mathcal{M}$ and the $\mathcal{M}$-actions on $\mathcal{N}$. Also regard $\mathcal{M}$ as a weak 2-category with a single object. There are two inclusion functors $l,r: \ \mathcal{M}\rightarrow\mathcal{B} $, mapping $\mathcal{M}$ identically to the endohom-categories of $L$ and $R$ respectively. Then the objects of $\text{w-}\Zc_{\mathcal{M}}(\mathcal{N}) $ are the lax natural transformations $l \rightarrow r$.
\end{remark}

 We need one more definition.

 \begin{definition}
 	\label{d:Z'}
 	Let $\mathcal{M} $ be a biclosed monoidal category, $\mathcal{M}^{op}$ a contragradient category as above. Then denote by $\Zc'_\mathcal{M}(\mathcal{M}^{op})$ the full subcategory of the \emph{weak center} that consists of objects such that the identity map $Id\in\Hom_\mathcal{M}(M,M)$ is mapped to same via \begin{equation}\label{stabilitycond}\Hom_\mathcal{M}(M,M)\simeq\Hom_\mathcal{M}(\id,M\la M)\to\Hom_\mathcal{M}(\id,M\ra M)\simeq \Hom_\mathcal{M}(M,M),\end{equation}
 	where the map in the middle is postcomposition with $\tau$ and the isomorphisms come from definitions of actions in $\mathcal{M}^{op}$.
 	This condition is called \emph{stability}.
 \end{definition}

 The next lemma shows that stability erases the difference between weak center and center.

 \begin{lemma}\label{weakstrong}
  In our current notation  $M \in \Zc'_\mathcal{M}(\mathcal{M}^{op}) $ implies that $M \in \Zc_\mathcal{M}(\mathcal{M}^{op})$.
 \end{lemma}

\begin{proof}
	By Yoneda Lemma, the condition \eqref{stabilitycond} implies that for any $V \in \mathcal{M}$ the chain of maps:
	$$\Hom_\mathcal{M}(V,M)\simeq\Hom_\mathcal{M}(\id,V\la M)\to\Hom_\mathcal{M}(\id,M\ra V)\simeq \Hom_\mathcal{M}(V,M) $$
	is identity.
	For any $T \in \mathcal{M}$ consider:
\begin{align*}  \Hom_\mathcal{M} (T\otimes V, M) &\simeq \Hom_\mathcal{M} (T, V \la M) \to \Hom_\mathcal{M} (T,  M\ra V) \simeq \Hom_\mathcal{M} (V\otimes T, M)\\
&\simeq \Hom_\mathcal{M} (V, T \la M) \to \Hom_\mathcal{M} (V,  M\ra T) \simeq \Hom_\mathcal{M} (T\otimes V, M).
\end{align*}
	By 	 the observation above and hexagon axiom this composition must be identity.
	Hence the map $\Hom_\mathcal{M} (T, V \la M) \to \Hom_\mathcal{M} (T,  M\ra V) $ is injective and the map $\Hom_\mathcal{M} (V, T \la M) \to \Hom_\mathcal{M} (V,  M\ra T)$ is surjective. Because this is true for any $T, V \in \mathcal{M}$, we can switch them and get that
	$$\Hom_\mathcal{M} (T, V \la M) \xrightarrow{\tau \circ} \Hom_\mathcal{M} (T,  M\ra V)  $$
	is bijective.
	This is true for any $T \in \mathcal{M}$, so again by the Yoneda lemma, we get that
	$\tau_{M,V}: V \la M \to M\ra V$ is an isomorphism.
\end{proof}

Now we want to use this result to construct a symmetric 2-contratrace, and hence Hopf-cyclic cohomology.

\begin{lemma}
	\label{reptrace}
	Let $\mathcal{M}$ be a biclosed linear monoidal category. The category of $\Zc'_{\mathcal{M}}(\mathcal{M}^{op})$ is equivalent to the category of representable left exact symmetric $2$-contratraces on $\mathcal{M}$ via $$M\longleftrightarrow\Hom_\mathcal{M}(-,M).$$
\end{lemma}

\begin{proof}
	Follows from the previous Lemma. See \cite{s} for details.
\end{proof}

\section{Hopf algebras}\label{s:Hopf}
We begin by considering the Hopf cyclic cohomology with contramodule coefficients in the case of Hopf algebras with \emph{invertible} antipode. Even though this case is well treated in \cite{contra}, our constructions become especially clear in this setting.  They serve a motivational purpose for the two cases of actual interest in this paper; namely that of quasi-Hopf algebras and Hopf algebroids.

Let $H$ be a Hopf algebra (we will always assume the bijectivity of the antipode in this section) over a field $k$. The category of $H$-modules, $\hmod$, is monoidal.  It is well-known that the category is also biclosed. We are going to include a  general proof (of a known fact \cite{BPO_closed, majid_closed, SS_closed}) that the modules over a quasi-Hopf algebra form a biclosed category in Section \ref{s:qHopf}.
For any $M,  N\in \hmod$,  we can define the left internal Hom by
\begin{equation}
\Hom^l(M, N)=\Hom_k(M, N),\quad h\cdot \varphi = h^1\varphi(S(h^2)-),
\end{equation}
and right internal Hom by
\begin{equation}
\Hom^r(M, N)=\Hom_k(M, N),\quad h\cdot \varphi = h^2\varphi(S^{-1}(h^1)-).
\end{equation}

Using the adjunctions for $M, V, W\in \hmod$:
$$ \ \Hom_\hmod (W\otimes V, M) \simeq \Hom_\hmod (W, \Hom^l(V,M)),$$ and
$$\Hom_\hmod (V\otimes W, M) \simeq \Hom_\hmod (W, \Hom^r(V,M)),$$
we can introduce the contragradient $\hmod$-bimodule category $\hmod^{op}$ as in Section \ref{s:General}.
Specifically, for $M\in\hmod^{op}$ and $V\in\hmod$, the action is given by:
\begin{equation}
M\ra V= \Hom^r(V, M), \quad ~~ \text{and} \quad ~~ V\la M=\Hom^l(V, M)  .
\end{equation}

The definition of a contramodule over a coalgebra was given first in \cite{EM_contra}, the definition of $aYD$-contramodules for a Hopf algebra was given in \cite{contra}, we recall it below.

\begin{definition}
	\label{Hopf:contra}
	Let $(H, \Delta, \varepsilon)$ be a coalgebra over $k$. A vector space $M$ is called a (right) $H$-contramodule, if there is a $k$-linear map $\mu: \Hom_k (H, M) \rightarrow M $, called contraaction, such that the following diagrams commute:
	\begin{center}
		\begin{tikzpicture}
		\matrix (m) [matrix of math nodes, row sep=4em, column sep=8em]
		{\Hom (H , \Hom(H , M))& & \Hom(H , M) \\
			\Hom((H\otimes H) , M)& \Hom(H,M) & M \\};
		\draw[->] (m-1-1) to node {$\Hom(H,\mu) $} (m-1-3);
		\draw[->] (m-1-3) to node {$\mu$} (m-2-3);
		\draw[->] (m-2-2) to node {$\mu$} (m-2-3);
		\draw[->] (m-1-1) to node {$\eta$} (m-2-1);
		\draw[->](m-2-1) to node {$\Hom(\Delta, M) $}(m-2-2);
		\end{tikzpicture}
	\end{center}
	where $\eta(f)(x\otimes y) = (f(x))(y)$ is a $k$-adjunction isomorphism;
	
	\begin{center}
		\begin{tikzpicture}
		\matrix (m) [matrix of math nodes, row sep=4em, column sep=8em]
		{\Hom(k , M) & \Hom(H , M) \\
			M \\};
		\draw[->] (m-1-1) to node {$\Hom(\varepsilon, M) $} (m-1-2);
		\draw[->] (m-1-2) to node {$\mu$} (m-2-1);
		\draw[->] (m-1-1) to node {$\simeq$} (m-2-1);
		
		\end{tikzpicture}
	\end{center}
	
A morphism of contramodules $(M, \mu) \rightarrow (N, \nu)$	is a map of vector spaces $t: M \rightarrow N$, such that the following diagram commutes:
\begin{center}
	\begin{tikzpicture}
	\matrix (m) [matrix of math nodes, row sep=4em, column sep=8em]
	{\Hom(H , M) & \Hom(H , N) \\
		M & N \\};
	\draw[->] (m-1-1) to node {$\Hom(H, t) $} (m-1-2);
	\draw[->] (m-1-2) to node {$\nu$} (m-2-2);
	\draw[->] (m-1-1) to node {$\mu$} (m-2-1);
	\draw[->] (m-2-1) to node {$t$} (m-2-2);
	\end{tikzpicture}
\end{center}

\end{definition}

\begin{definition}
	Let $H$ be a Hopf algebra over $k$. A $k$-vector space $M$ is called an  anti-Yetter-Drinfeld contramodule ($aYD$-contramodule), if
	\begin{itemize}
		\item $M$ is a left $H$-module,
		\item $M$ is a  $H$-contramodule,
		\item For all $h\in H$ and any linear map $f\in \Hom(H, M)$ there is a compatibility condition:
		\begin{equation}
			\label{Hopf:aydeq}h(\mu(f))=\mu(h^2 f(S(h^3)-h^1)).
		\end{equation}
	\end{itemize}

It is called stable ($saYD$-contramodule), if for all $m\in M$ we have $\mu(r_m)=m$ where $r_m(h)=hm$.

A morphism of $aYD$-contramodules $M \rightarrow N$ is simultaneously a morphism of $H$-modules and $H$-contramodules.
\end{definition}

Because the category $\hmod$ is a biclosed monoidal, we can define the contragradient $\hmod$-bimodule category $\hmod^{op}$ as in Section \ref{s:General}, and describe its center.

\begin{proposition}
	\label{Hopf:aydprop}
	The category of $aYD$-contramodules for a Hopf algebra $H$, with an invertible antipode $S$, is isomorphic to $\Zc_\hmod(\hmod^{op})$.
	
\end{proposition}

\begin{proof}
	Let $(M, \mu)$ be an $aYD$-contramodule, for any $H$-module $V$ define $\tau:V\la M\to M\ra V$ by $$\tau(\phi)(v)=\mu(h\mapsto\phi(hv)),$$
	where $\phi \in \Hom^l(V,M), \ v \in V, h\in H$.
	
	Define $\theta: M\ra V\to V\la M$ by $$\theta(\phi)(v)=\mu(h\mapsto\phi(S^{-1}(h)v)),$$
	where $\phi \in \Hom^r(V,M), \ v \in V, h\in H$.
	
	Then ($x\in H$)
	 \begin{align*}(\theta\circ\tau(\phi))(v)&=\mu(h\mapsto(\tau\phi)(S^{-1}(h)v))=\mu(h\mapsto\mu(x\mapsto\phi(xS^{-1}(h)v)))\\&=\mu(h\mapsto\phi(h^2S^{-1}(h^1)v))=\mu(h\mapsto \varepsilon(h)\phi(v))=\phi(v).
	\end{align*}
	And similarly $\tau\theta\phi=\phi$, so that $\tau$ is an isomorphism (of vector spaces so far).
	
	Now for $x \in H$ let us compare $(x\tau\phi)(v)$ to $(\tau x\phi)(v)$ using \eqref{Hopf:aydeq}:
	 \begin{align*}(x\tau\phi)(v)&=x^2\mu(h\mapsto\phi(hS^{-1}(x^1)v))=\mu(h\mapsto x^3\phi(S(x^4)hx^2 S^{-1}(x^1)v))\\
	&=\mu(h\mapsto x^2\phi(S(x^3)h\varepsilon(x^1)v))=\mu(h\mapsto x^1\phi(S(x^2)hv))=(\tau x\phi)(v)
	\end{align*} so that $\tau$ is an isomorphism in $\hmod$.
	
	Let $f\in\Hom_H(V,V')$ then $$\tau(\phi\circ f)(v)=\mu(h\mapsto \phi(f(hv)))=\mu(h\mapsto\phi(hf(v)))=(\tau\phi\circ f)(v)$$ so that $\tau$ is functorial in $V$.
	
	Next, to verify the commutativity of \begin{equation}\label{Hopf:assoc}\xymatrix{W\la(V\la M)\ar[r]^{Id\la\tau}& W\la(M\ra V)\ar[r]^{Id}& (W\la M)\ra V\ar[r]^{\tau\ra Id}  &(M\ra W)\ra V\ar[d]\\
		(W\ot V)\la M\ar[u]\ar[rrr]^{\tau} & & & M\ra(W\ot V) \\
	}\end{equation} is to check that taking $f$ along the second row to $w\ot v\mapsto\mu(h\mapsto f(h^1 w\ot h^2 v))$ is the same as taking it the long way around, which is $w\ot v\mapsto\mu(x\mapsto\mu(h\mapsto f(xw\ot hv)))$.  Those two are equal by the contraaction condition.  Finally, to show that $k\la M\to M\ra k$ is the identity observe that if, for $m\in M$ we define a function, also called $m: k\to M$, as follows:  $m(c)=cm$, then \begin{equation}
	\label{Hopf:unit}\tau m(c)=\mu(h\mapsto m(hc))=\mu(h\mapsto\epsilon(h)cm)=cm.\end{equation}
	
Furthermore, if $g:M\to M'$ is a map of $aYD$-contramodules then $$(g\circ\tau\phi)(v)=g(\mu(h\mapsto\phi(hv)))=\mu(h\mapsto g(\phi(hv)))=\tau(g\circ\phi)(v)$$ so that $g$ is a morphism in $\Zc_\hmod(\hmod^{op})$.

	What has been shown so far is that if $M$ is an $aYD$-contramodule, then $(M,\tau)\in\Zc_\hmod(\hmod^{op})$ and any $g:M\to M'$ a morphism of $aYD$-contramodules induces a morphism between the corresponding central elements.

	Conversely, given a central element  $(M,\tau)\in\Zc_\hmod(\hmod^{op})$, take $\mu$ to be the composition of $\tau:\Hom^l(H,M)\to\Hom^r(H,M)$ and $ev_1:\Hom^r(H,M)\to M$.    Observe that the considerations \eqref{Hopf:assoc} and \eqref{Hopf:unit} are of the if and only if kind, so that $\mu$ is a contraaction.  Furthermore, the $aYD$ condition is satisfied.  More precisely, let $x\in H$ and $f\in\Hom(H,M)$, note that $\tau(f(-h))=(\tau f)(-h)$ since $r_h:H\to H$ is a morphism in $\hmod$, then
	\begin{align*}
	\mu(x^2 f(S(x^3)-x^1))&=ev_1\tau(x^2 f(S(x^3)-x^1))=ev_1\tau((x^2\cdot f)(-x^1))\\ &=ev_1 x^2\cdot\tau(f(-x^1))=ev_1 x^2\cdot(\tau f)(-x^1)=ev_1 x^3(\tau f)(S^{-1}(x^2)-x^1)\\
	&=x^3(\tau f)(S^{-1}(x^2)x^1)=x(\tau f)(1)=x ev_1\tau f=x\mu(f).\\
	\end{align*}
	
	Finally, suppose that $f:M\to M'$ is a map in the center, then we have $$\xymatrix{H\la M\ar[d]_{Id\la f}\ar[r]^{\tau}& M\ra H\ar[d]_{f\ra Id}\ar[r]^{ev_1}& M\ar[d]_f\\
		H\la M'\ar[r]^\tau & M'\ra H\ar[r]^{ev_1} & M\\
	}$$ where the left square commutes by the centrality of $f$ and the right square commutes obviously.  Since the top row is $\mu_M$ and the bottom row is $\mu_{M'}$ so $f$ is a map of $aYD$-contramodules.
	
	Note that the constructions of central elements from $aYD$-contramodules and vice versa are inverses of each other. Indeed, the direction $\mu \rightarrow \tau \rightarrow \mu$  is obvious. The second direction  $\tau \rightarrow \mu \rightarrow \tau$  follows from the observation that $\tau_H(f(-h))=(\tau_H f)(-h)$. Thus we have proved the isomorphism of categories as claimed.
\end{proof}

\begin{remark}
	\label{weak_vs_strong}
Note that $\theta$, the inverse of $\tau$ defined in the above proof is easily written down.  This is not the case in the instance of quasi-Hopf algebras nor Hopf algebroids thus making it necessary to deal with the weak center, at least until the imposition of stability.  We do not know if in the presence of an invertible antipode the weak center coincides with the center for quasi-Hopf algebras and Hopf algebroids.  The notion of stability, required for cyclic cohomology, allows us to side-step the issue.
\end{remark}
In Definition \ref{d:Z'} we let $\Zc'_\hmod(\hmod^{op})$ be the full subcategory of the weak center that consists of objects such that the identity map $Id\in\Hom_H(M,M)$ is mapped to same via
 \begin{equation}
\label{Hopf:stabilitycond}\Hom_H(M,M)\simeq\Hom_H(\id,M\la M)\to\Hom_H(\id,M\ra M)\simeq \Hom_H(M,M).
\end{equation}  Recall that Lemma \ref{weakstrong} shows that $\Zc'_\hmod(\hmod^{op})$ is actually a full subcategory of the strong center, and so the middle map above is an isomorphism.

We immediately obtain the following Corollary of Proposition \ref{Hopf:aydprop}.

\begin{corollary}\label{Hopf:saydcor}
	The category of $saYD$-contramodules for $H$ is equivalent to $\Zc'_\hmod(\hmod^{op})$.
\end{corollary}
\begin{proof}
	Let $M$ be an $aYD$-contramodule, then to ensure that the condition \eqref{Hopf:stabilitycond} is satisfied, we need that $$\tau(Id)(m)=\mu(h\mapsto hm)=\mu(r_m)=m$$ which is exactly the $saYD$ condition.  On the other hand, if $M$ is a central element that satisfies \eqref{Hopf:stabilitycond}, then it also satisfies that the chain of isomorphisms \begin{equation}\label{Hopf:fullstabcond}\Hom_H(V,M)\simeq\Hom_H(\id,V\la M)\simeq\Hom_H(\id,M\ra V)\simeq \Hom_H(V,M) \end{equation} is identity for any $V\in\hmod$.  Take $V=H$ and note that $r_m\in\Hom_H(H,M)$ so that \eqref{Hopf:fullstabcond} implies that $\tau(r_m)=r_m$ and so $\mu(r_m)=ev_1\tau(r_m)=ev_1(r_m)=m$, i.e., $M$ is a stable $aYD$-contramodule.
\end{proof}

Let's prove an easy lemma which will motivate the definition of the $aYD$-condition in the case of Hopf algebroids.

\begin{lemma}
	Let $H$ be a Hopf algebra with an invertible antipode, $M$ be a left $H$-module and $H$-contramodule. Then equation \eqref{Hopf:aydeq} is equivalent to:
	\begin{equation}
	\label{Hopf:aydeq2}
	h^2 (\mu(f(-S^{-1}(h^1))))=\mu(h^{1} f(S(h^{2})-))
	\end{equation}
	for all $h \in H$, and $f \in \Hom_k(H,M)$.
\end{lemma}

\begin{proof}
	Assume that \eqref{Hopf:aydeq} holds. Then consider the left hand side of \eqref{Hopf:aydeq2}:
	\begin{align*}
	h^2 (\mu(f(-S^{-1}(h^1)))) & \stackrel{\text{(i)}}{=} \mu (h^3f(S(h^4)-h^2S^{-1}(h^1)))\\
	& = \mu (h^2f(S(h^3)-\varepsilon(h^1)))  = \mu(h^{1} f(S(h^{2})-)) ,
	\end{align*} where (i) is equation \eqref{Hopf:aydeq}.

	Now assume that \eqref{Hopf:aydeq2} holds. Then:
		\begin{align*}
		\mu(h^2 f(S(h^3)-h^1)) & \stackrel{\text{(i)}}{=} h^3\mu (f(-S^{-1}(h^2)h^1))\\
		& = h^2\mu (f(-\varepsilon(h^1)))  = h\mu(f) ,
		\end{align*} where (i) is equation \eqref{Hopf:aydeq2}.

\end{proof}

\section{Quasi-Hopf algebras}\label{s:qHopf}
 Let us remind the reader of all the necessary definitions following \cite{Drinfeld}. In this section $k$ is a field.
 \begin{definition}
A quasi-bialgebra is a collection $(A, \Delta, \varepsilon, \Phi)$, where $A$ is an associative $k$-algebra with unity,  $\Delta: A \rightarrow A\otimes A$ and  $\varepsilon: A \rightarrow k $ are homomorphisms of algebras, $\Phi \in A \otimes A\otimes A$ is an invertible elements,  such that the following equalities hold:
\begin{equation}
\label{coassoc}
(\text{id} \otimes \Delta) \ (\Delta (a) ) = \Phi \cdot ( (\Delta \otimes \text{id}) \ (\Delta(a) ) ) \cdot \Phi^{-1} \qquad \forall a \in A
\end{equation}

\begin{equation}
\label{phi}
\lbrack (id \otimes id \otimes \Delta)(\Phi) \rbrack \cdot \lbrack (\Delta \otimes id \otimes id)(\Phi) \rbrack = (1 \otimes \Phi) \cdot \lbrack (id \otimes \Delta \otimes id)(\Phi) \rbrack \cdot (\Phi \otimes 1)
\end{equation}

\begin{equation}
\label{one}
(\varepsilon \otimes id ) \ (\Delta (a)) =  a, \qquad (id \otimes \varepsilon ) \ (\Delta (a)) =  a, \qquad \forall a \in \mathcal{A}
\end{equation}

\begin{equation}
\label{unass}
id \otimes \varepsilon \otimes id \ (\Phi) = 1 \otimes 1
\end{equation}
 \end{definition}

 \begin{remark}
 	\label{Sweedler}
 	In this paper we will use the Sweedler notation. Let's denote
 	 \begin{equation}
 	 \Phi = X \otimes Y \otimes Z , \\
 	  	 \end{equation}
 	  \begin{equation}
 	  \Phi^{-1} = P \otimes Q \otimes R , \\
 	  \end{equation}
 	here we mean the summation. In particular, the equality \eqref{coassoc} can be written as:
 	$$a^1\otimes a^{21}\otimes a^{22} = Xa^{11}P\otimes Ya^{12}Q \otimes Z a^2 R. $$

 \end{remark}

We are interested in the category of left $A$-modules $_A\mathcal{M}$. It was proved in \cite{Drinfeld} that this category is monoidal if a tensor product of two left $A$-modules $M$ and $N$ is defined by the same formula as in the case of a bialgebra:
\begin{equation}
M \otimes N = M \otimes_k N, \quad a\cdot (m\otimes n) = a^1m\otimes a^2n.
\end{equation}

 The associativity morphism is no longer trivial as it was in the case of a bialgebra. If one sets the associativity morphism
$(M \otimes N)\otimes L  \rightarrow M \otimes (N\otimes L)$ to be the image of $\Phi$ in $\text{End}_k(M \otimes N \otimes L)$, then it becomes an isomorphism of left $A$-modules by \eqref{coassoc}. Consider $k$ as an $A$-module by $a \cdot 1 = \varepsilon(a) 1 $ as in the bialgebra case. Then one defines a morphism $\lambda_M: k \otimes M \rightarrow M $ as the usual morphism of $k$-modules. Then $\lambda_M$ becomes an $A$-module morphism by \eqref{one}. Similarly one can define a morphism $\rho_M: M\otimes k \rightarrow M $. So $k$ is a unit of a monoidal category $_A\mathcal{M}$.

\begin{remark}
	The equality \eqref{phi} is equivalent to the pentagon axiom. Associativity and unit in the category respect each other by \eqref{unass}.
	
\end{remark}

Recall the definition of a quasi-Hopf algebra from \cite{Drinfeld}.

\begin{definition}
	Let $(H, \Delta, \varepsilon, \Phi)$ be a quasi-bialgebra. Then it is called a quasi-Hopf algebra if there exist $\alpha, \ \beta \in H $ and anti-automorphism $S: H \rightarrow H$, such that
	\begin{equation}
	\label{alpha}
	S(h^1)\alpha h^2 = \epsilon(h) \alpha ,
	\end{equation}
	\begin{equation}
	\label{beta}
	h^1 \beta S(h^2) = \epsilon(h) \beta .
	\end{equation}
	If one keeps notation as in Remark \ref{Sweedler}, then there should be equalities
	\begin{equation}
	\label{evcoev}
	X \beta S(Y) \alpha Z  = 1,
	\end{equation}
	\begin{equation}
	\label{coevev}
	S(P) \alpha Q \beta R  = 1.
	\end{equation}
	
\end{definition}

\begin{remark}
	We want to emphasize that the antipode $S$ in the definition above is assumed to be invertible.
\end{remark}

It was shown in \cite{Drinfeld}, that the category of $H$-modules for a quasi-Hopf algebra $H$, that are finite dimensional as $k$-vector spaces, is rigid. The more general result, that without the finite dimensionality condition  the category of $H$-modules is closed, was proved in \cite{BPO_closed, majid_closed, SS_closed}. To fix notation  and for the convenience of the reader we are going to include a sketch of a proof here.

\begin{proposition}
	\label{internalhom}
	The category $\hmod$ of left modules over a quasi-Hopf algebra $H$ is a biclosed monoidal category.
\end{proposition}

 \begin{proof}
 	For any $M,  N\in \hmod$,  we can define the left internal Hom by
 	\begin{equation}
 	\label{lefthom}
 	\Hom^l(M, N)=\Hom_k(M, N),\quad h\cdot \varphi = h^1\varphi(S(h^2)-).
 	\end{equation}
 	Define also the evaluation morphism:
 	\begin{equation}
 	\label{ev_l}
 	\text{ev}^l:\ \Hom^l(M, N)\otimes M \rightarrow N,  	
 	\quad  \varphi \otimes m \mapsto X ( \varphi (S(Y)\alpha Z m) ) .
 	\end{equation}
 	
 	Notice that the morphism $\text{ev}^l$ is a morphism of left $H$-modules. Let's define an adjunction map:
 	\begin{equation}
 	 \zeta^l: \ \Hom_\hmod (M\otimes N, L) \rightarrow \Hom_\hmod (M, \Hom^l(N,L))
 	\end{equation}
 	by the rule:
 	\begin{equation}
 	f \mapsto (m \mapsto f(Pm\otimes Q\beta S(R) -)).
 	\end{equation}
 	
 	Observe that $\zeta^l(f)$ is a morphism of $H$-modules. Construct also a map in the opposite direction:
 	\begin{equation}
 	\eta^l: \ \Hom_\hmod (M, \Hom^l(N,L))\rightarrow  \Hom_\hmod (M\otimes N, L)
 	\end{equation}
 	by the rule:
 	\begin{equation}
 	 \eta^l(g) = \text{ev}^l \circ(g\otimes id).
 	  	\end{equation}
 It is a morphism 	of $H$-modules as a composition of such. It will be useful to write the formula:
 $$ \eta^l(g)(m\otimes n) = X(g(m)(S(Y)\alpha Z n)). $$
 	
 	Note that $\zeta^l$ and $\eta^l $ are mutually inverse and the category $\hmod$ is left closed.
\\

	For any $M,  N\in \hmod$,  we can define the right internal Hom by
	\begin{equation}
	\label{righthom}
	\Hom^r(M, N)=\Hom_k(M, N),\quad h\cdot \varphi = h^2\varphi(S^{-1}(h^1)-).
	\end{equation}
	Define also the evaluation morphism:
	\begin{equation}
	\label{ev}
	\text{ev}^r:\ M \otimes \Hom^r(M, N) \rightarrow N,  	
	\quad  m \otimes \varphi  \mapsto R ( \varphi (S^{-1}(Q)S^{-1}(\alpha) P m) ) .
	\end{equation}
	
	Let's define an adjunction map:
	\begin{equation}
	\label{r_adj}
	\zeta^r: \ \Hom_\hmod (N\otimes M, L) \rightarrow \Hom_\hmod (M, \Hom^r(N,L))
	\end{equation}
	by the rule:
	\begin{equation}
	f \mapsto (m \mapsto f(YS^{-1}(\beta)S^{-1}(X) - \otimes Zm)).
	\end{equation}
	
	As above $\zeta^r(f)$ is a morphism of $H$-modules.
	
	Construct also a map in the opposite direction:
	\begin{equation}
	\eta^r: \Hom_\hmod (M, \Hom^r(N,L)) \rightarrow  \Hom_\hmod (N\otimes M, L)
	\end{equation}
	by the rule:
	\begin{equation}
	\eta^r(g) = \text{ev}^r \circ (id \otimes g) ,
	\end{equation}
 the formula is: $  \eta^r(g)(n\otimes m) = R(g(m)(S^{-1}(Q)S^{-1}(\alpha)P n))$.

 Analogously to the case of left internal Hom one sees that $\eta^r$ and $\zeta^r$ are mutually inverse. This means that the category $\hmod$ is also right closed, and thus biclosed.  \end{proof}

\section{Anti-Yetter-Drinfeld contramodules for a quasi-Hopf algebra}\label{s:qHopf-ayd-contra}

We are going to define the anti-Yetter-Drinfeld contramodules using the categorical approach from Section \ref{s:General}.  The category $\hmod$ is  biclosed monoidal. Using the adjunctions for $M, V, W\in \hmod$:
$$ \ \Hom_\hmod (W\otimes V, M) \simeq \Hom_\hmod (W, \Hom^l(V,M)),$$ and
$$\Hom_\hmod (V\otimes W, M) \simeq \Hom_\hmod (W, \Hom^r(V,M)),$$
proved in Proposition \ref{internalhom}, we can introduce the contragradient $\hmod$-bimodule category $\hmod^{op}$ as in Section \ref{s:General}.
 Specifically, for $M\in\hmod^{op}$ and $V\in\hmod$, the action is given by:
\begin{equation}
M\ra V= \Hom^r(V, M), \quad ~~ \text{and} \quad ~~ V\la M=\Hom^l(V, M)  .
\end{equation}

For the Hopf-case we proved Proposition \ref{Hopf:aydprop}, that the center $\Zc_\hmod(\hmod^{op}) $ is the same as anti-Yetter Drinfeld contramodules. We want to generalize this fact to the quasi-Hopf case, i.e., we need to change the notion of a contramodule due to the noncoassociativity phenomenon. We will only prove that $aYD$-contramodules coincide with the weak center, but it is enough for our purposes, as was explained in Section \ref{s:General}.  Unlike the Hopf algebra setting there are two types of contramodules.  More precisely, there are two, in this case different, ways of unraveling the definition of the center into formulas.

We organize the material in the similar way to \cite{majid}.

\subsection{Anti-Yetter-Drinfeld contramodules I}

We begin with a lemma explaining the origin of the $aYD$-condition.

\begin{lemma}
	\label{aYD-contramodule lemma I}
	Let $M \in \hmod^{op}$. Natural transformations $\tau \in \text{Nat} (id \la M, M \ra id)$ are in 1-1 correspondence with $k$-linear maps $\mu: \Hom_k (H,M) \rightarrow M  $,  such that
	\begin{equation}
		\label{aYD-contramoduleI}
		h^2 (\mu(f(-S^{-1}(h^1))))=\mu(h^{1} f(S(h^{2})-)).
	\end{equation}
\end{lemma}

\begin{proof}
Consider $f \in \Hom_k(H,M)$, induce the $H$-module structure corresponding to the left internal homomorphism. Given a morphism $\tau_H : H \la M \rightarrow M \ra H $, define: $\mu(f) = \tau_H(f)(1) $, which we denote by $ev_1(\tau_H(f))$. Note that this is not a right ``categorical" evaluation! Let $x\in H$ and $f\in\Hom(H,M)$, observe that $\tau_H(f(-h))=(\tau_H f)(-h)$ since $r_h:H\to H$ is a morphism in $\hmod$. Now we can check the formula:
\begin{align*}
\mu(h^{1} f(S(h^{2})-))&=ev_1\tau_H (h^{1} f(S(h^{2})-)) =ev_1\tau_H(h \cdot f )=ev_1 h \cdot\tau_H(f)\\ & =ev_1 h^2 (\tau_H f)(S^{-1}(h^1)-)=h^2 (\tau_H f)(S^{-1}(h^1))\\ & = h^2 ev_1 \tau_H(f(-S^{-1}(h^1)))=  h^2 (\mu(f(-S^{-1}(h^1)))) .
\end{align*}

Conversely, given the map $\mu$, for any $V \in \hmod $ one can define the map $\tau_V: V \la M \rightarrow M \ra V $ by the rule $f \mapsto \mu(x\mapsto f(x-))$. It is a morphism of $H$-modules by \eqref{aYD-contramoduleI}.

Given an $H$-homomorphism $\varphi: V \rightarrow W $, consider the diagram:

 $$\xymatrix{\Hom^l(W,M) \ar[d]_{-\circ \varphi}\ar[r]^{\tau_W}& \Hom^r(W,M) \ar[d]_{-\circ \varphi }\\
 	\Hom^l(V,M) \ar[r]^{\tau_V} & \Hom^r(V,M). }$$
It commutes, so we have naturality.

These correspondences are mutually inverse. The only nontrivial direction follows from the observation that $\tau_H(f(-h))=(\tau_H f)(-h)$.
\end{proof}

Now we need to find a replacement for a contramodule condition. Let's  consider a hexagon axiom of the weak center:

\begin{small}
\begin{center}
	\begin{tikzpicture}
	\matrix (m) [matrix of math nodes, row sep=4em, column sep=8em]
	{V \la (W \la M)& V \la (M \ra W) & (V \la M) \ra W \\
		(V\otimes W) \la M& M \ra (V\otimes W) & (M \ra V) \ra W. \\};
	\draw[->] (m-1-1) to node {$id \la \tau_W $} (m-1-2);
	\draw[->] (m-1-2) to node {} (m-1-3);
	\draw[->] (m-1-3) to node {$\tau_V \ra id$} (m-2-3);
	\draw[->] (m-2-2) to node {} (m-2-3);
	\draw[->] (m-1-1) to node {} (m-2-1);
	\draw[->](m-2-1) to node {$\tau_{V\otimes W} $}(m-2-2);
	\end{tikzpicture}
\end{center}
\end{small}

 There are three associativity isomorphisms and we need explicit formulas for them:
\begin{equation*}
V \la (W \la M) \rightarrow (V \otimes W) \la M, \quad (v\mapsto (w\mapsto f_v(w))) \mapsto (v\otimes w \mapsto Xf_{S(Z)v}(S(Y)w));
\end{equation*}
\begin{equation*}
V \la (M \ra W) \rightarrow (V \la M) \ra W, \quad (v\mapsto (w\mapsto f_v(w))) \mapsto (w\mapsto (v\mapsto Yf_{S(Z)v}(S(X)w)));
\end{equation*}
\begin{equation*}
M \ra (V \otimes W)  \rightarrow (M \ra V) \ra W  , \quad (v\otimes w \mapsto  f(v\otimes w) ) \mapsto (w\mapsto (v\mapsto Zf(S(Y)v\otimes S(X)w)) .
\end{equation*}

Assume that $V = H = W$ in the hexagon diagram. After evaluating the resulting morphisms at $1$, we  obtain the equality:
\begin{align}
\label{quasi-contraI}
\mu (h \mapsto Y \mu(x \mapsto f_{S(Z)h}(xS(X)))) = Z \mu (h \mapsto f_{S(Z)h^1 S(Y)}(S(Y) h^2 S(X) ) ).
\end{align}

\begin{remark}
	Notice, that if $H$ is a Hopf algebra, then the condition \eqref{quasi-contraI} becomes:
	$$\mu (h \mapsto \mu (x\mapsto f_h(x))) = \mu (h \mapsto f_{h^1}(h^2)), $$
	which is exactly the right contramodule condition as used in \cite{contra} (see Definition \ref{Hopf:contra}).
\end{remark}

The unitality condition $\tau_k = id_M$  gives  the  following relation. For $f \in \Hom^l(k,M)$,
\begin{equation}
\label{contramodule_unitI}
f(1) = \mu (x \mapsto \varepsilon(x)f(1)).
\end{equation}

\begin{definition}
	Let $H$ be a quasi-Hopf algebra. A pair $(M, \mu)$, where $M$ is a left $H$-module and $\mu: \Hom(H,M)  \rightarrow M$ is a $k$-linear map,  is called a left-right anti-Yetter-Drinfeld contramodule of type I, if it satisfies the equalities \eqref{aYD-contramoduleI},  \eqref{quasi-contraI} and \eqref{contramodule_unitI}.
	
	A morphism of two $aYD$-contramodules $(M, \mu) \rightarrow (M', \mu')$ is an $H$-morphism $f: M\rightarrow M'$, which commutes with $\mu$ and $\mu'$.
\end{definition}

\begin{theorem}\label{aydpropI}
	The category of $aYD$-contramodules of type I for a quasi-Hopf algebra $H$ is isomorphic to $\text{w-}\Zc_\hmod(\hmod^{op})$.
\end{theorem}

\begin{proof}
	We have seen that an object in the center gives an $aYD$-contramodule. Assume that we have a morphism of two central objects: $f: (M,\tau) \rightarrow (M', \tau')$.	Then we have:
	$$\xymatrix{H\la M\ar[d]_{Id\la f}\ar[r]^{\tau_H}& M\ra H\ar[d]_{f\ra Id}\ar[r]^{ev_1}& M\ar[d]_f\\
		H\la M'\ar[r]^{\tau'_H} & M'\ra H\ar[r]^{ev_1} & M'\\
	}$$ where the left square commutes by the centrality of $f$ and the right square commutes obviously.  Since the top row is $\mu$ and the bottom row is $\mu'$ so $f$ is a map of $aYD$-contramodules.
	
	Conversely, let $(M, \mu)$ be an $aYD$-contramodule. By Lemma \ref{aYD-contramodule lemma I}, there is a natural transformation $\tau: id \la M \rightarrow M \ra id$. Formula \eqref{quasi-contraI} gives the hexagon axiom for $\tau$. Equality \eqref{contramodule_unitI} gives the condition $\tau_k = id$.

	We observed  that the constructions of central elements from $aYD$-contramodules and vice versa are inverses of each other, thus we have proved the isomorphism of categories.
\end{proof}

\begin{remark}
	\label{weak_center_quasi}
	For a Hopf algebra with an invertible antipode  the center  $\Zc_\hmod(\hmod^{op})$ and the weak center $\text{w-}\Zc_\hmod(\hmod^{op})$ coincide (one can easily write the formula for the $\tau^{-1}$ - see the proof of Proposition \ref{Hopf:aydprop} and Remark \ref{weak_vs_strong}). In the quasi-Hopf case we can neither prove nor disprove the analogous fact.
	
\end{remark}

\subsection{Anti-Yetter-Drinfeld contramodules II}

We can introduce $aYD$-contramodules of type II using the actual categorical evaluation and adjunctions for internal homomorphisms.

\begin{lemma}
	\label{aYD-contramodule lemma}
	Let $M \in \hmod^{op}$. Natural transformations $\tau \in \text{Nat} (id \la M, M \ra id)$ are in 1-1 correspondence with $k$-linear maps $\nu: \Hom_k (H,M) \rightarrow M  $,  such that
	\begin{equation}
	\label{aYD-contramodule}
	h\cdot (\nu(f))=\nu(h^{21} f(S(h^{22})-h^1)).
	\end{equation}
\end{lemma}

\begin{proof}
Consider the morphism $\tau_H : H \la M \rightarrow M \ra H $. Using the adjunction it  is in one-to-one correspondence with the morphism: $\hat{\tau}_H: H\otimes \Hom^l(H,M) \rightarrow M $. Let's define $\nu$ by the rule:
$$\nu(f) = \hat{\tau}_H (1\otimes f)$$
for any $f \in \Hom_k(H,M)$ with the necessary $H$-module structure.

The right action of $H$ on itself is a morphism in $\hmod$, so, because $\hat{\tau}$ is natural, we have the formula: $\hat{\tau}(x\otimes f) = \hat{\tau}(1\otimes f(-\cdot x)) $. Using this we get the following equalities:
\begin{align*}
h \cdot \nu(f) &= h \cdot \hat{\tau}(1\otimes f) =\hat{\tau}(h\cdot (1\otimes f)) =
\hat{\tau}(h^1 \otimes h^{21}f(S(h^{22}) - ))\\
 &=\hat{\tau}(1 \otimes h^{21}f(S(h^{22}) - h^1)) = \nu(h^{21} f(S(h^{22})-h^1)).
\end{align*}

Conversely, given a $k$-linear map	$\nu:  \Hom_k (H,M) \rightarrow M$, equip the vector space $\Hom_k (H,M)$ with an $H$-module structure in the following way:
$$h \cdot f = h^{21} f(S(h^{22})-h^1). $$
Denote this $H$-module by $\Hom^\Delta (H,M) $.
Then \eqref{aYD-contramodule} implies that $\nu$ is a morphism in the category $\hmod$. Consider also a map:
$$\theta_V: V\otimes \Hom^l(V,M) \rightarrow \Hom^\Delta (H,M), \quad v\otimes f \mapsto (x \mapsto f(xv)). $$
One can easily check that it is  a morphism in $\hmod$. Now define $\hat{\tau}_V $ as a composition $\nu \circ \theta_V $. Clearly, $\hat{\tau}_V $ is natural in $V$.  These correspondences are mutually inverse. For a given $v \in V$ consider a morphism $H \rightarrow V$ by the rule $h \mapsto h \cdot v$. By naturality of $\hat{\tau}_V$, it is uniquely  defined from $\nu$.
\end{proof}

\begin{remark}
	It will be useful to write explicitly  the reconstruction formula of $\tau$ for a given  $M \in \hmod$ and a map   $\nu : \Hom_k(H,M)\rightarrow M$ satisfying \eqref{aYD-contramodule}. For $V \in \hmod$, the map $\tau_V: \Hom^l(V,M)\rightarrow \Hom^r(V,M)$ is constructed by the rule:
	\begin{equation}
	\label{mu}
	f \mapsto \nu (x \mapsto Z^1 f(S(Z^2)xYS^{-1}(\beta) S^{-1}(X) -    )).
		\end{equation}
	
	 We can see it from the proof of Lemma \ref{aYD-contramodule lemma} and the adjuction.
	
	 Analogously, one can observe that $\nu = \text{ev}^r_1 \circ \tau_H$, where $\text{ev}^r_1$ is a composition of maps $(\text{unit}\otimes id): \Hom^r(H,M) \rightarrow H \otimes \Hom^r(H,M) $ and $\text{ev}^r$.
	 \end{remark}

As for type I from the hexagon axiom we get:
\begin{align}
\begin{split}
\label{quasi-contra}
\nu(h \mapsto Z^1Y\nu(x\mapsto &Z^1f_{S(Z)\kappa(h)}(\kappa(x) S(X) )))\\ = &
Z\nu (h \mapsto Z^1f_{S(Z)(\kappa(h))^1S(Y)  }( S(Y)(\kappa(h))^2S(X)   ) ) ,
\end{split}
\end{align}
where we used the notation $\kappa(x) = S(Z^2)xYS^{-1}(\beta) S^{-1}(X) $.

The unitality condition $\tau_k = id_M$ applied with the rule \eqref{mu} gives  the  following relation. For $f \in \Hom^l(k,M)$,
$ f(1) = \nu (x \mapsto \varepsilon(S(Z^2)xYS^{-1}(\beta) S^{-1}(X))  Z^1 f( 1  )).$

For a quasi-Hopf algebra the equality $\varepsilon \circ S = \varepsilon $ holds (for the proof see \cite{Drinfeld}). So we can simplify:
$$\varepsilon(S(Z^2)xYS^{-1}(\beta) S^{-1}(X))  Z^1 = Z \varepsilon(Y)  \varepsilon(X) \varepsilon(\beta) \varepsilon(x) = \varepsilon(\beta) \varepsilon(x). $$
Finally, we have:
\begin{equation}
\label{contramodule_unit}
f(1) = \nu (x \mapsto  \varepsilon(\beta)  \varepsilon(x)   f( 1  )).
\end{equation}

\begin{definition}
	Let $H$ be a quasi-Hopf algebra. A pair $(M, \nu)$, where $M$ is a left $H$-module and $\nu: \Hom(H,M)  \rightarrow M$ is a $k$-linear map,  is called a left-right anti-Yetter-Drinfeld contramodule of type II, if it satisfies the equalities \eqref{aYD-contramodule},  \eqref{quasi-contra} and \eqref{contramodule_unit}.
\end{definition}

\begin{theorem}\label{aydprop}
The category of $aYD$-contramodules of type II for a quasi-Hopf algebra $H$ is isomorphic to $\text{w-}\Zc_\hmod(\hmod^{op})$.

\end{theorem}

\begin{proof}
Similar to Theorem \ref{aydpropI}.
\end{proof}

\begin{remark}
	As we mentioned before, the difference between $aYD$-contramodules of type I and the ones of type II comes from the difference between the naive evaluation and the categorical evaluation in the category of $H$-modules for a quasi-Hopf algebra $H$. The two structure maps $\mu$ and $\nu$ are not the same. For the convenience of the reader we provide the explicit formulas relating the two structures.
	
	Given $(M, \mu)$, an $aYD$-contramodule of type I,  one keeps the $H$-module structure on $M$ unchanged, but  as an $aYD$-contramodule of type II $(M, \nu_\mu)$ is given by the formula:
	\begin{equation}
	\nu_\mu (f) = R \mu (h \mapsto f(h S^{-1}(Q) S^{-1}(\alpha) P )), \quad \text{for any} \ f \in \Hom(H,M) .
	\end{equation}
	
	Conversely, given $(M, \nu)$, an $aYD$-contramodule of type II, one can define $(M, \mu_\nu)$, an $aYD$-contramodule of type I by the formula:
	\begin{equation}
	\mu_\nu (f) =  \nu (h \mapsto (Z\cdot f)(h Y S^{-1}(\beta) S^{-1}(X)  )), \quad \text{for any} \ f \in \Hom(H,M) .
	\end{equation}
	
\end{remark}

 \subsection{Cohomology theory with contramodule coefficients}

 Let's begin with the discussion of stability.  Recall the Definition \ref{d:Z'} of $\Zc'_\hmod(\hmod^{op})$.   We want to write an explicit formula for this condition. First for a given $aYD$-contramodule $M$ of type I and an element $m \in M$ we can define a  $k$-linear map: $r'_m: H \rightarrow M$, by the rule:
 \begin{equation}
 x \mapsto  \beta    x  S^{-1} (Q) S^{-1} (\alpha) P  m      .
 \end{equation}

 \begin{definition}
 	Let $M$ be an 	$aYD$ contramodule of type I. It is called stable ($saYD$-contramodule), if for all $m \in M$ we have
 	\begin{equation}
 	\label{d:stable}
 	R \mu(r'_m)=m .
 	\end{equation}
 \end{definition}

 We immediately obtain the following Corollary of Theorem \ref{aydprop}.

 \begin{corollary}\label{saydcor}
 	The category of $saYD$-contramodules for $H$ is equivalent to $\Zc'_\hmod(\hmod^{op})$.
 \end{corollary}

 \begin{proof}
 	The condition \eqref{stabilitycond} that $id_M$ maps to the same under the chain of maps is literally the equation \eqref{d:stable} if we use the explicit formulas for adjunctions and $\tau$ and notice that $\varepsilon(P)Q\beta S(R) = \beta$
 \end{proof}

Using Lemma \ref{reptrace}, we can  construct Hopf-cyclic cohomology for quasi-Hopf algebras with $saYD$-contramodule coefficients.

\section{Hopf algebroids}\label{s:Hopfroids}

First we will give all the necessary definitions following \cite{Bohm}. Let $k$ be a commutative ring and $R$ an algebra over $k$ (not necessary commutative). A ring $A$ together with a ring map $\eta: R \rightarrow A$ is called an $R$-ring. Categorically, $R$-rings are monoids in the category of $R$-bimodules.

We will use the following easy observation. Giving the ring map $\eta: R\otimes_k R^{op} \rightarrow A$ is equivalent to giving two (so called, source and target) maps: $s := \eta(-\otimes 1_R): R \rightarrow A $ and $t := \eta(1_R\otimes -): R^{op} \rightarrow A $.

Similarly $R$-corings are defined as comonoids in the category of $R$-bimodules. So an $R$-coring is a triple $(C, \Delta, \epsilon)$,  where $C$ is an $R$-bimodule, and $\Delta: C \rightarrow C\otimes_R C$ and $\epsilon: C \rightarrow R$ are $R$-bimodule maps satisfying coassiativity and counit conditions.

Bialgebroids are generalization of bialgebras, but now algebra and coalgebra structures are defined in different monoidal categories.

\begin{definition}
Let $R_l$ be a $k$-algebra. A left $R_l$-bialgebroid $\mathcal{B}_l$ consists of $R_l\otimes_k R_l^{op}$-ring $(\mathcal{B}_l, s_l, t_l)   $	 and $R_l$-coring $(\mathcal{B}_l, \Delta_l, \epsilon_l) $ on the same $k$-module $\mathcal{B}_l$, such that
\begin{enumerate}
	\item the bimodule structure in the  $R_l$-coring $(\mathcal{B}_l, \Delta_l, \epsilon_l) $ is related to the  $R_l\otimes_k R_l^{op}$-ring  structure via
	$$ r \cdot b \cdot r' := s_l(r)t_l(r') b, \quad r,r' \in R_l, \ b \in \mathcal{B}_l $$
	\item the coproduct  $\Delta_l: \mathcal{B}_l \rightarrow \mathcal{B}_l\otimes_{R_l} \mathcal{B}_l $ corestricts to a $k$-algebra map from $\mathcal{B}_l$ to $\mathcal{B}_l \times^l_{R_l} \mathcal{B}_l $. Here $\mathcal{B}_l \times^l_{R_l} \mathcal{B}_l$ is the Takeuchi product defined by:
	$$ \mathcal{B}_l \times^l_{R_l} \mathcal{B}_l := \{\sum_{i} b_i \otimes_{R_l} b_i'\ \arrowvert \sum_{i} b_i t_l(r) \otimes_{R_l} b_i' = \sum_{i} b_i  \otimes_{R_l} b_i's_l(r) \ \forall r \in R_l  \} .$$
	
	\item The left counit is a left character of the $R_l$-ring $(\mathcal{B}_l, s_l)$, i.e., $$\epsilon_l(bb') = \epsilon_l(b s_l(\epsilon_l b')) = \epsilon_l(b t_l(\epsilon_l b')), \ \forall b,b' \in \mathcal{B}_l$$
\end{enumerate}
\end{definition}
So a bialgebroid is a quintuple of data $(\mathcal{B}_l, R_l, s_l, t_l, \Delta_l, \epsilon_l)$, but we will usually write it as just $\mathcal{B}_l$.

Analogously one defines right bialgebroids.
\begin{definition}
 Let $R_r$ be a $k$-algebra. A right $R_r$-bialgebroid $\mathcal{B}_r$ consists of  $R_r\otimes_k R_r^{op}$-ring $(\mathcal{B}_r, s_r, t_r) $ and $R_r$-coring $(\mathcal{B}_r, \Delta_r, \epsilon_r) $ on the same $k$-module $\mathcal{B}_r$. The following condition holds:
 \begin{enumerate}
 \item $R_r$-bimodule structure in the coring $(\mathcal{B}_r, \Delta_r, \epsilon_r) $  is given by:
$$r \cdot b \cdot r' :=  bt_r(r)s_r(r') , \quad r,r' \in R_r, \ b \in \mathcal{B}_r . $$

\item The coproduct  $\Delta_r: \mathcal{B}_r \rightarrow \mathcal{B}_r\otimes_{R_r} \mathcal{B}_r $ corestricts to a $k$-algebra map from $\mathcal{B}_r$ to $\mathcal{B}_r \times^r_{R_r} \mathcal{B}_l $. Here one changes the Takeuchi product:
$$ \mathcal{B}_r \times^r_{R_r} \mathcal{B}_r := \{\sum_{i} b_i \otimes_{R_r} b_i'\ \arrowvert \sum_{i} s_r(r) b_i  \otimes_{R_r} b_i' = \sum_{i} b_i  \otimes_{R_r} t_r(r)b_i' \ \forall r \in R_r  \} .$$

\item The right counit is a right character of the $R_r$-ring $(\mathcal{B}_r, s_r)$, i.e., $$\epsilon_r(bb') = \epsilon_r(s_r(\epsilon_r b)b') = \epsilon_r(t_r(\epsilon_r b)b'), \ \forall b,b' \in \mathcal{B}_r.$$

\end{enumerate}
\end{definition}

	We will use the following Sweedler's notation for coproducts:
	\begin{equation}
	\Delta_l(b) = b_1\otimes_{R_l} b_2,  \qquad  \Delta_r(b) = b^1\otimes_{R_r} b^2.
	\end{equation}

Let us consider a left bialgebroid $\mathcal{B}_l$. In particular  $\mathcal{B}_l$ is a ring, so we can consider a category of left modules over it $ _{\mathcal{B}_l}\mathcal{M}$. Any left $\mathcal{B}_l$-module $M$ is also a $R_l$-bimodule via
$$r \cdot m \cdot r' := s_l(r)t_l(r') m, \quad r,r' \in R_l, \ m \in M. $$
Given $M,N \in  {_{\mathcal{B}_l}}\mathcal{M}$, we have an $R_l$-bimodule $M\otimes_{R_l} N $. It can be supplied with the left $\mathcal{B}_l$-module structure via the left coproduct $\Delta_l $:
$$b \cdot (m\otimes_{R_l} n ) = b_1 m \otimes_{R_l} b_2n, \  \forall b \in \mathcal{B}_l $$

\begin{remark}
 Schauenburg proved \cite{sch} that for a $R_l\otimes_k R_l^{op} $-ring $(\mathcal{B}_l, s_l, t_l) $, the following structures are equivalent:
 \begin{itemize}
 	\item structure of a left algebroid on $\mathcal{B}_l$
 	 \item a monoidal structure on the category $ _{\mathcal{B}_l}\mathcal{M}$, such that the forgetful functor $ _{\mathcal{B}_l}\mathcal{M} \rightarrow {_{R_l}\mathcal{M}_{R_l}}$ is strictly monoidal.
 	  \end{itemize}

 Similarly, any right $\mathcal{B}_r $-module $M$ is an $R_r$-bimodule via
 $$r \cdot m \cdot r' :=  m s_r(r')t_r(r), \quad r,r' \in R_r, \ m \in M. $$
 If $M, N \in \mathcal{M}_{\mathcal{B}_r}$, we can induce a right  $\mathcal{B}_r $-module structure on $M\otimes_{R_r} N $ via right comultiplication $\Delta_r$.
Given an $R_r\otimes_k R_r^{op} $-ring $(\mathcal{B}_r, s_r, t_r) $, the structures are equivalent:
 \begin{itemize}
 	\item structure of a right algebroid on $\mathcal{B}_r$
 	\item a monoidal structure on the category of right $\mathcal{B}_r$-modules $\mathcal{M}_{\mathcal{B}_r}$, such that the forgetful functor $ \mathcal{M}_{\mathcal{B}_r} \rightarrow {_{R_r}\mathcal{M}_{R_r}}$ is strictly monoidal.
 \end{itemize}
\end{remark}

Finally we can recall the definition of a Hopf algebroid.

\begin{definition}
\label{def:algebroid}
	A Hopf algebroid is given by a triple $(\mathcal{H}_l, \mathcal{H}_r, S )$, \\where $\mathcal{H}_l = (\mathcal{H}_l, R_l, s_l, t_l, \Delta_l, \epsilon_l)$ is a left $R_l$-bialgebroid, \\$\mathcal{H}_r = (\mathcal{H}_r, R_r, s_r, t_r, \Delta_r, \epsilon_r)$ is a right $R_r$-bialgebroid on the same $k$-algebra $\mathcal{H}$\\ and $S: \mathcal{H} \rightarrow \mathcal{H} $ is a $k$-module map satisfying the following axioms:
	\begin{enumerate}
	\item $s_l\circ \epsilon_l\circ t_r = t_r, \quad  s_r\circ \epsilon_r\circ t_l = t_l, \quad t_l\circ \epsilon_l\circ s_r = s_r, \quad t_r\circ \epsilon_r\circ s_l = s_l. $
	\item Mixed coassociativity: $(\Delta_l \otimes_{R_r} id_{\mathcal{H}}) \Delta_r =  (id_{\mathcal{H}} \otimes_{R_l} \Delta_r  ) \Delta_l$ \\
		and $(\Delta_r \otimes_{R_l} id_{\mathcal{H}}) \Delta_l =  (id_{\mathcal{H}} \otimes_{R_r} \Delta_l  ) \Delta_r$.
	\item For $r \in R_l, \ r' \in R_r$ and $h \in \mathcal{H}$, we have $S(t_l(r)ht_r(r')) = s_r(r')S(h)s_l(r). $
	\item $m \circ (S\otimes_{R_l} id_\mathcal{H}) \circ \Delta_l = s_r\circ \epsilon_r$ and $m \circ (id_\mathcal{H} \otimes_{R_r} S ) \circ \Delta_r = s_l\circ \epsilon_l$,  where $m$ is a multiplication in the  $k$-algebra $\mathcal{H}$.
	\end{enumerate}
\end{definition}

\begin{remark}
	From property $(1)$ it follows that $R_l$ is isomorphic to $R_r^{op}$.
\end{remark}

Mixed coassociativity can be written in Sweedler's notation as:
$$ (h^1)_1 \otimes_{R_l} (h^1)_2   \otimes_{R_r} h^2 =    h_1 \otimes_{R_l} (h_2)^1   \otimes_{R_r}   (h_2)^2    $$
and
$$(h_1)^1 \otimes_{R_r} (h_1)^2   \otimes_{R_l} h_2 =    h^1 \otimes_{R_r} (h^2)_1   \otimes_{R_l}   (h^2)_2.$$

Let $\mathcal{H} = (\mathcal{H}_l, \mathcal{H}_r, S )$ be a Hopf algebroid. The map $S$ is called the antipode. From now on we will always assume that the antipode $S$ is invertible. As was proved in \cite[Proposition 4.4]{Bohm}, the antipode is an antihomomorphism of the ring $\mathcal{H} $.

\begin{remark}
For a Hopf algebroid $\mathcal{H}$ and any $h \in \mathcal{H}$ we have the following identities:
\begin{equation}  S^{-1}(h_2) h_1 = t_r \epsilon_r(h), \quad  \quad h^2 S^{-1}(h^1) = t_l\epsilon_l(h),  \end{equation}
see \cite[Lemma 4]{kow}.

For $r \in R_l$ we have
\begin{equation}
\label{kow1}
t_r\epsilon_rt_l(r) = S^{-1} (t_l(r)), \quad \quad s_r\epsilon_rs_l(r) = S(s_l(r)),
\end{equation}
see \cite[1.8]{kow}.
\end{remark}

\begin{lemma}
	\label{module S-1}
	For a Hopf algebroid $\mathcal{H} $, and $h \in \mathcal{H}$, $r \in R_l, r' \in R_r$ one has:
	\begin{equation}
	t_r(r') S^{-1}(h) t_l(r) = S^{-1}(s_l(r) h s_r(r'))
	\end{equation}
\end{lemma}
\begin{proof}
	Because the antipode is a unital map, property (3) of Definition \ref{def:algebroid} implies: $S(t_l(r)) = s_l(r) $ and $S(t_r(r')) = s_r(r') $.  Because the antipode is an antihomomorphism, we have:
	$S^{-1}(s_l(r)h s_r(r')) = S^{-1}(s_r(r')) S^{-1}(h) S^{-1}(s_l(r)) = t_r(r') S^{-1}(h) t_l(r).$
\end{proof}

A structure of a left module over the Hopf algebroid $\mathcal{H}$  is a structure of a left module over the underlying $k$-algebra $\mathcal{H}$. We want to study the category of left $\mathcal{H}$-modules $_\mathcal{H} \mathcal{M} $. Firstly recall that it is a monoidal category, because $\mathcal{H}_l $ is a left bialgebroid, and a module structure is given by left comultiplication $\Delta_l$. We will denote the monoidal structure simply by $\otimes$.

If $M,  N $ are $R$-bimodules, we will denote by  $\Hom(M,N)_R$ a morphism of right $R$-modules, and by $\Hom_R(M,N)$ a morphism of left $R$-modules. We need the following easy lemma:
\begin{lemma}
	\label{right-left}
If $M,  N \in \hmodalg$, then $f \in \Hom(M,N)_{R_l}$ is equivalent to $f \in \Hom_{R_r}(M,N) $. Similarly, $f \in \Hom(M,N)_{R_r}$ is equivalent to $f \in \Hom_{R_l}(M,N) $.
\end{lemma}
\begin{proof}
If $f \in \Hom(M,N)_{R_l}$, then for any $r \in R_l$ and $m \in M$, $f(t_l(r)m) = t_l(r)f(m) $.  Now let's consider $a \in R_r $, then  $f(s_r(a)m) = f(t_l \epsilon_l s_r (a)m)$ (here we used property (1) in Definition \ref{def:algebroid}). By $R_l$-right linearity, $f(t_l \epsilon_l s_r (a)m) = t_l \epsilon_l s_r (a)f(m) $. Switching back we get $s_r(a)f(m) $.  Three other implications are proved analogously.
\end{proof}

 The closeness of the category of left modules over a Hopf algberoid was proved in \cite{sch_closed}. This result is crucial for most of the constructions in the paper. To fix notation and for the convenience of the reader we are going to include a proof here.

 \begin{proposition}
 	\label{internalhom_algebroids}
 	The category of left modules over a Hopf algebroid $_\mathcal{H} \mathcal{M} $ is a biclosed monoidal category.
 \end{proposition}
 \begin{proof}
 	For any $M,  N \in \hmodalg$,  we can define the left internal Hom by
 	\begin{equation}
 		\label{lefthomalg}
 		\Hom^l(M, N)=\Hom(M, N)_{R_l},\quad h\cdot \varphi = h^1\varphi(S(h^2)-).
 	\end{equation}

 	Notice that we used the structure of right comultiplication here.
 	
 	First of all, we need to prove that this is well defined. For any $a \in R_r $ consider, by Lemma \ref{right-left} and property (3) in Definition \ref{def:algebroid}:
 	$$
 	h^1 s_r(a) \varphi (S(h^2) -)   = h^1 \varphi (s_r(a) S(h^2) - ) = h^1 \varphi(S(h^2 t_r(a)) - ).
 	$$		
 	So at least the combination of symbols $h^1\varphi(S(h^2)-)$ makes sense. 	 Now let's check  if $\varphi \in \Hom(M, N)_{R_l}$ implies that $h^1\varphi(S(h^2)-) $ is also a morphism of right $R_l$-modules. To do that we observe:
 	 	$$
 	t_l(r)h^1 \otimes_{R_r} h^2 = s_r\epsilon_rt_l(r)h^1 \otimes_{R_r} h^2 = h^1 \otimes_{R_r} t_r\epsilon_rt_l(r)h^2
 	= h^1 \otimes_{R_r} S^{-1}(t_l(r))h^2 	
 	 			  		$$ by property (1) in Definition \ref{def:algebroid}, right Takeuchi property, and \eqref{kow1} respectively.
 	
 Finally, using that $S$ is an antihomomorphism, we have:
  $$t_l(r)h^1 \varphi(S(h^2) - )  = h^1 \varphi(S(h^2) t_l(r) - ) ,$$		
  so   		$h^1\varphi(S(h^2)-) \in \Hom(M, N)_{R_l}$.

 		An adjunction map
 		\begin{equation}
 			\zeta^l: \ \Hom_\hmodalg (M\otimes N, L) \rightarrow \Hom_\hmodalg (M, \Hom^l(N,L))
 		\end{equation}
 	is defined similarly to the case of Hopf algebras:
 		\begin{equation}
 			f \mapsto (m \mapsto f(m\otimes  -)).
 		\end{equation}
 		
 		Take $f \in \Hom_\hmodalg (M\otimes N, L) $. Observe that $\zeta^l(f)$ is a morphism of $\mathcal{H}$-modules:
 		\begin{align*}
 			h \cdot \zeta^l(f) (m) & = h \cdot f(m\otimes  -) = h^1 ( f(m\otimes S(h^2) - ))  = f((h^1)_1 m \otimes (h^1)_2 S(h^2) -) \\ & \stackrel{\text{(i)}}{=} f(h_{1}m \otimes (h_2)^1 S((h_2)^2)   -)   \stackrel{\text{(ii)}}{=} f(h_1 m \otimes_{R_l} s_l \epsilon_l(h_2) -)\\ 	
 			 & = f(t_l \epsilon_l(h_2)h_1 m \otimes_{R_l} -)  \stackrel{\text{(iii)}}{=}	\zeta^l(f) (hm). \end{align*}  Where (i) is mixed coassociativity, (ii) property (4) of Definition \ref{def:algebroid}, (iii) $R_l$-bimodule structure and counit.
		
		Define  the evaluation morphism:
 	\begin{equation}
 	 	 \text{ev}^l:\ \Hom^l(M, N)\otimes M \rightarrow N,  	
 	\quad  \varphi \otimes m \mapsto \varphi ( m)  .
 	\end{equation}
 	
 	Notice that the morphism $\text{ev}^l$ is the morphism of left $\mathcal{H}$-modules. Indeed,
 	\begin{align*}
 	\text{ev}^l (h \cdot (\varphi \otimes m ) ) &= 	\text{ev}^l (h_1 \cdot \varphi \otimes h_2 \cdot m) = (h_1)^1\varphi(S((h_1)^2)  h_2 m)\\
 	&= h^1\varphi(S((h^2)_1) (h^2)_2 m) =  h^1\varphi( s_r\epsilon_r(h^2) m)  \\
	&= h^1 s_r\epsilon_r(h^2) \varphi(m) 	= h \cdot \text{ev}^l(\varphi \otimes m).
 	\end{align*}
 	
 	Using the evaluation we can construct a map that is to be the inverse of $\zeta^l$:
 	\begin{equation}
 	\eta^l: \ \Hom_\hmodalg (M, \Hom^l(N,L))\rightarrow  \Hom_\hmodalg (M\otimes N, L)
 	\end{equation}
 	by the rule:
 	\begin{equation}
 	 \eta^l(g) = \text{ev}^l \circ(g\otimes id).
 	  	\end{equation}  It is a morphism 	of $\mathcal{H}$-modules as a composition of such.  To see that these maps are mutually inverse we can forget about additional structure and use the usual tensor-hom adjunction theorem for $R_l$-bimodules.
 		 Hence we proved that the category $\hmodalg$ is left closed.
 		\\
 		
 		For any $M,  N\in \hmodalg$,  we can define the right internal Hom by
 		\begin{equation}
 		\label{righthomalg}
 		\Hom^r(M, N)=\Hom_{R_l}(M, N)\quad h\cdot \varphi = h^2\varphi(S^{-1}(h^1)-),
 		\end{equation}
 	unlike the case of left internal homs, the underlying structure is a morphism of left $R_l$-modules.
 	
 		As for the left hom, we need to check that this is well defined. For any $a \in R_r $ consider:
 		$$
 		h^2 t_r(a) \varphi (S^{-1}(h^1) -)   = h^2 \varphi (t_r(a) S^{-1}(h^1) - ) 	
 		 = h^2 \varphi(S^{-1}(h^1 s_r(a)) - )$$ by Lemma \ref{right-left} and Lemma \ref{module S-1} respectively.
 		So the notation $h^2\varphi(S^{-1}(h^1)-) $ makes sense.
 		
 		Now  we check that if $\varphi \in \Hom_{R_l}(M, N)$ then $h^2\varphi(S^{-1}(h^1)-) $ is also a morphism of left $R_l$-modules. First, make an observation:
 		$$h^1 \otimes_{R_r} s_l(r) h^2 = h^1 \otimes_{R_r} t_r \epsilon_r s_l (h^2 )= s_r\epsilon_r s_l(r) h^1 \otimes_{R_r} h^2= S(s_l(r))h^1 \otimes_{R_r} h^2$$ by property (1) in Definition \ref{def:algebroid}, right Takeuchi property, and \eqref{kow1} respectively. Because $S^{-1}$ is an antihomomorphism, we have:
 		$$s_l(r) h^2\varphi(S^{-1}(h^1)-)  = h^2 \varphi(S^{-1}(h^1) s_l(r) - ) ,$$		
 		so   		$h^2\varphi(S^{-1}(h^1)-) \in \Hom_{R_l}(M, N)$.

		Let's define an adjunction map:
 		\begin{equation}
 		\label{r_adj alg}
 		\zeta^r: \ \Hom_\hmodalg (N\otimes M, L) \rightarrow \Hom_\hmodalg (M, \Hom^r(N,L))
 		\end{equation}
 		by the rule:
 		\begin{equation}
 		f \mapsto (m \mapsto f( - \otimes m)).
 		\end{equation}

Define also the evaluation morphism:
	\begin{equation}
	\text{ev}^r:\ M \otimes \Hom^r(M, N) \rightarrow N,  	
	\quad  m \otimes \varphi  \mapsto   \varphi ( m) .
	\end{equation}

Check that the morphism $\text{ev}^r$ is the morphism of left $\mathcal{H}$-modules:
 	\begin{align*}
 	\text{ev}^r (h \cdot (m \otimes \varphi ) ) &= 	\text{ev}^r (h_1 \cdot m \otimes h_2 \cdot \varphi ) = (h_2)^2 \varphi(S^{-1}((h_2)^1)  h_1 m)\\
 	&= h^2\varphi(S^{-1}((h^1)_2) (h^1)_1 m) =  h^2\varphi( t_r\epsilon_r(h^1) m)  \\
	&= h^2 t_r\epsilon_r(h^1) \varphi(m) 	= h \cdot \text{ev}^r(m \otimes \varphi).
 	\end{align*}
		
As above one can construct the inverse of $\zeta^r$ using the evaluation, we skip the details.  So $\hmodalg$ is right closed.
 	
 \end{proof}

At the end of the section let's observe the following useful fact.

\begin{lemma}
	\label{rights} Let $\varphi \in \Hom^l(M,N)$, and $\psi\in \Hom^r(M,N) $. Then:
  \begin{align*}	t_l(r) \cdot \varphi = \varphi(s_l(r)-), \quad s_l(r) \cdot \varphi = s_l(r)\varphi(-) \\
  s_l(r) \cdot \psi = \psi(t_l(r)-), \quad t_l(r) \cdot \psi = t_l(r)\psi(-). \end{align*}
\end{lemma}
\begin{proof}
To prove this, we need to recall that $\Delta_r$ is both $R_l$- and $R_r$-bimodule map \cite{Bohm}. So
$\Delta_r(t_l(r)) = 1^1 \otimes_{R_r} t_l(r) 1^2 $ and $\Delta_r(s_l(r)) = s_l(r)1^1 \otimes_{R_r}  1^2 $.
Then the proof of the first formula is as follows:
$$
t_l(r) \cdot \varphi = 1^1 \varphi (S(t_l(r) 1^2 ) - )  = 1^1 \varphi (S( 1^2 )s_l(r) - )
= \varphi (s_l(r) -)$$ where the last equality is due to $\Delta_r$ being unital, while the second formula is obtained thus:
$$s_l(r) \cdot \varphi = s_l(r)1^1 \varphi (S(1^2 ) - ) = s_l\varphi (-) .$$
The two others are proved similarly.	
	\end{proof}

\section{Anti-Yetter-Drinfeld contramodules for Hopf algebroids}\label{s:Hopfroid-ayd-contra}

In this section we are going to describe the coefficients for the Hopf-cyclic cohomology theory for Hopf algebroids. We will discuss only aYD contramodules and proceed similarly to the case of quasi-Hopf algebras. The category $\hmodalg$ is biclosed monoidal.
Using the adjunctions for $M, V, W\in \hmodalg$:
$$ \ \Hom_\hmodalg (W\otimes V, M) \simeq \Hom_\hmodalg (W, \Hom^l(V,M)),$$ and
$$\Hom_\hmodalg (V\otimes W, M) \simeq \Hom_\hmodalg (W, \Hom^r(V,M)),$$
proved in Proposition \ref{internalhom_algebroids}, we can introduce the contragradient $\hmodalg$-bimodule category $\hmodalg^{op}$ as in Section \ref{s:General}. Specifically, for $M\in\hmodalg^{op}$ and $V\in\hmodalg$, the action is given by:
\begin{equation}
M\ra V= \Hom^r(V, M), \quad ~~ \text{and} \quad ~~ V\la M=\Hom^l(V, M).
\end{equation}

We want to understand  the center $\Zc_\hmodalg(\hmodalg^{op}) $.

First of all let us recall the notion of contramodules over bialgebroids (as it was  defined in \cite{BoS}, generalizing the classical contramodules for coalgebras \cite{EM_contra}).

\begin{definition}
\label{contra_alg}
A right contramodule over a left $R_l$-bialgebroid $\mathcal{B}_l$ is a right $R_l$-module $M$ together with a right $R_l$-module map:
$$
\mu: \Hom(\mathcal{B}_l,M)_{R_l} \to M,
$$
called the  contraaction, such that the diagrams commute:

\begin{small}
\begin{center}
 	\begin{tikzpicture}
  	\matrix (m) [matrix of math nodes, row sep=4em, column sep=8em]
 	{\Hom(\mathcal{B},  \Hom(\mathcal{B}, M)_{R_l})_{R_l} &  & \Hom(\mathcal{B},M)_{R_l} \\
 		 \Hom(\mathcal{B} \otimes_{R_l} \mathcal{B}, M)_{R_l} & \Hom(\mathcal{B} , M)_{R_l} & M \\};
 	\draw[->] (m-1-1) to node {$\Hom(\mathcal{B},\mu)_{R_l}$} (m-1-3);
 	\draw[->] (m-1-3) to node {$\mu$} (m-2-3);
 	\draw[->] (m-2-2) to node {$\mu$} (m-2-3);
 	\draw[->] (m-1-1) to node {$\cong$} (m-2-1);
 	\draw[->](m-2-1) to node {$\Hom(\Delta_l, M)_{R_l}$}(m-2-2);
 	\end{tikzpicture}
 \end{center}
\end{small}
and

\begin{center}
	\begin{tikzpicture}
	\matrix (m) [matrix of math nodes, row sep=3em, column sep=5em]
	{\Hom(R_l, M)_{R_l} & \Hom(\mathcal{B}, M)_{R_l}  \\
		M \\};
	\draw[->] (m-1-1) to node {$\Hom(\epsilon_l, M)_{R_l}$} (m-1-2);
	\draw[->] (m-1-1) to node {$\cong  $} (m-2-1);
	\draw[->](m-1-2) to node {$\mu $}(m-2-1);
	\end{tikzpicture}
\end{center}

where  $\mathcal{B}_l $  is  a $R_l$-bimodule via $r \cdot b \cdot r' = s_l(r)t_l(r')b $. The right $R_l$-module structure on $\Hom(\mathcal{B}_l,M)_{R_l}$ is defined by $ f(-) \cdot r = f(r\cdot -) = f(s_l(r)-)$.  The isomorphism  $\Hom(\mathcal{B},  \Hom(\mathcal{B}, M)_{R_l})_{R_l} \rightarrow \Hom(\mathcal{B} \otimes_{R_l} \mathcal{B}, M)_{R_l}$ is simply a tensor-hom adjunction.   \end{definition}

Let's write the formulas explicitly; the first commutative diagram gives:
\begin{equation}\label{contraaction assoc}
\mu (x \mapsto \mu( y \mapsto \varphi(x \otimes y))) = \mu (h \mapsto \varphi(h_1 \otimes_{R_l} h_2)), \quad\text{for} \  \varphi \in  \Hom(\mathcal{B} \otimes_{R_l} \mathcal{B}, M)_{R_l}.\end{equation}

The second diagram is:
\begin{equation}
\label{contraaction unit}
\mu (x \mapsto  m \cdot \epsilon_l(x)) =    m,  \quad \text{for any}  \ m \in M .\end{equation}

As was observed in \cite{BoS} or \cite{kow_contra}, we can equip a right contramodule $M$ with a structure of a left $R_l$-module via:
$$ r \cdot m = \mu(x \mapsto m \cdot \epsilon_l(x s_l(r))), \quad \text{for any}  \ m \in M, r \in R_l. $$
Under this action the morphism $\mu$ is an $R_l$-bimodule map, if we define a left $R_l$-module structure on $\Hom(\mathcal{B}_l,M)_{R_l} $ by $r \cdot f(-) = f(-s_l(r))  $ (see \textit{loc.cit.})  Let us state the definition of an anti-Yetter-Drinfeld contramodule over a Hopf algebroid. Notice that the definition is slightly different from the one used in \cite{kow_contra}.

\begin{definition}
	\label{def:aYD_contra_alg}
	An anti-Yetter-Drinfeld ($aYD$) contramodule $M$ over a Hopf algebroid $\mathcal{H}$ is a left $\mathcal{H} $-module and a right $\mathcal{H}_l$-contramodule, such that both underlying $R_l$-bimodule structures coinside: $r \cdot m \cdot r' = s_l(r)t_l(r')m $, and the following $aYD$-condition holds:
	\begin{equation}
	\label{aYD-contramodule_alg2}
		h^2 (\mu(f(-S^{-1}(h^1))))=\mu(h^{1} f(S(h^{2})-)).
		\end{equation}
	
\end{definition}

Let's explicitly write the condition that $\mu$ is $R_l $-bilinear. Right $R_l$-linearity is:
\begin{equation}
\label{right mu}
\mu (f(s_l(r) - )) = t_l(r)\mu(f(-)), \quad  f \in \Hom (\mathcal{H},M)_{R_l}, \ r \in R_l .
\end{equation}
Left $R_l$-linearity gives:
\begin{equation}
\label{left mu}
\mu (f( - s_l(r))) = s_l(r)\mu(f(-)), \quad  f \in \Hom (\mathcal{H},M)_{R_l}, \ r \in R_l .
\end{equation}

\begin{remark}
	Let us observe that all expressions in equation \eqref{aYD-contramodule_alg2} are well defined, if $\mu$ is $R_l$-bilinear. We just need to check the linearity of the left hand side.  Namely, we have by Lemma \ref{module S-1} and Lemma \ref{right-left} with formula \eqref{right mu} respectively: $$
	h^2 (\mu(f(-S^{-1}(h^1 s_r(a))))) =    h^2 (\mu(f(-t_r(a)S^{-1}(h^1))))
	= h^2t_r(a)\mu (f (-  S^{-1}(h^1) ) ).$$
\end{remark}

We are going to show that the category of $aYD$-contramodules is  the weak center of a bimodule category $\hmodalg^{op} $.

\begin{lemma}
	\label{aYD-contramodule_alg lemma}
	Let $M \in \hmodalg^{op}$. Natural transformations $\tau \in \text{Nat} (id \la M, M \ra id)$ are in 1-1 correspondence with right $R_l$-module  maps $\mu: \Hom (\mathcal{H},M)_{R_l} \rightarrow M  $ (formula \eqref{right mu}), which are also left $R_l$-module maps (formula \eqref{left mu}), such that the property \eqref{aYD-contramodule_alg2} holds.
\end{lemma}

\begin{proof}

	Consider $f \in \Hom(\mathcal{H},M)_{R_l}$, induce the $\mathcal{H}$-module structure corresponding to the left internal homomorphism. Given a morphism $\tau_\mathcal{H} : \mathcal{H} \la M \rightarrow M \ra \mathcal{H} $, define: $\mu(f) = \tau_\mathcal{H}(f)(1) $, which we denote by $ev_1(\tau_\mathcal{H}(f))$.
	Let $x \in \mathcal{H}$ and $f\in\Hom(\mathcal{H},M)_{R_l}$, observe that $\tau_\mathcal{H}(f(-h))=(\tau_\mathcal{H} f)(-h)$ since $r_h:\mathcal{H}\to \mathcal{H}$ is a morphism in $\hmodalg$.

	First, let's check that $\mu$ satisfies \eqref{right mu}:
	\begin{align*}
	\mu(f(s_l(r)-)) & \stackrel{\text{(i)}}{=} \mu (t_l(r) \cdot f(-))= ev_1(\tau (t_l(r)\cdot f ))\stackrel{\text{(ii)}}{=} ev_1(t_l(r)\cdot\tau ( f))\\
	& \stackrel{\text{(i)}}{=} ev_1 (t_l(r)(\tau ( f)) )= t_l(r)\tau(f)(1) = t_l(r) \mu(f),
	\end{align*} where (i) is Lemma \ref{rights} and (ii) holds since  $\tau$ is $\mathcal{H}$-linear.
	
	Now we check that $\mu$ satisfies \eqref{left mu}:
\begin{align*}
\mu (f( - s_l(r))) & = ev_1 \tau (f( - s_l(r)))= ev_1 \tau (f)( - s_l(r))  \\
&= \tau (f)( s_l(r))\stackrel{\text{(i)}}{=} s_l(r)\tau (f)(1)= s_l(r) \mu (f),	
\end{align*} where (i) is due to the left $R_l$-linearity of $\tau$.

	Now we can check the $aYD$-condition:

	\begin{align*}\mu(h^{1} f(S(h^{2})-)) &=ev_1\tau(h^{1} f(S(h^{2})-)) =ev_1\tau((h \cdot f)(-)) =ev_1 h\cdot\tau(f(-))\\
	&=ev_1 h^2(\tau f)(S^{-1}(h^1)-) =h^2(\tau f)(S^{-1}(h^1))\\
		& = h^2ev_1\tau (f (- S^{-1}(h^1)) ) = h^2 \mu(f (- S^{-1}(h^1)) ).\end{align*}

	Conversely, given a map $\mu: \Hom(\mathcal{H}, M)_{R_l} \to M$, satisfying \eqref{aYD-contramodule_alg2}, \eqref{right mu} and \eqref{left mu}.
	For any $V \in \hmodalg $ one can define a map:
	 $$\tau_V: V \la M \rightarrow M \ra V, \ \text{by the rule} \ f \mapsto \mu(x\mapsto f(x-)).$$
	
We check that it is well-defined. First of all, for a given $v \in V$ and $f \in \Hom(V,M)_{R_l} $, a map $x \mapsto f(xv) $ is in $\Hom(\mathcal{H}, M)_{R_l}$. Indeed, $f(t_l(r)xv) = t_l(r)f(xv) $, because $f$ is $R_l$-linear.
	
	 Then we need to understand that the target, a map $v \mapsto \mu(x\mapsto f(xv))$, is in $\Hom_{R_l}(V,M)$. It follows directly from formula \eqref{left mu}.

	Finally, we should check, that the constructed map $\tau$ is a morphism of $\mathcal{H}$-modules:
	
	 \begin{align*}
	 (h\cdot \tau (f))(v)=& \,\, h^2 \tau f(S^{-1}(h^1)v)=h^2\mu(x\mapsto f(xS^{-1}(h^1)v))\\\stackrel{\eqref{aYD-contramodule_alg2}}{=}&\mu (x \mapsto (h \cdot f) (v) )=(\tau (h \cdot f))(v).
	\end{align*}

	Given an $H$-homomorphism $\varphi: V \rightarrow W $, consider a diagram:
	
	$$\xymatrix{\Hom^l(W,M) \ar[d]_{-\circ \varphi}\ar[r]^{\tau_W}& \Hom^r(W,M) \ar[d]_{-\circ \varphi }\\
		\Hom^l(V,M) \ar[r]^{\tau_V} & \Hom^r(V,M). }$$
	It commutes, so we have naturality.
	
	These correspondences are mutually inverse. The only nontrivial direction follows from the observation that $\tau_{\mathcal{H}}(f(-h))=(\tau_{\mathcal{H}} f)(-h)$.
\end{proof}

Finally, we can formulate the following:

\begin{proposition}\label{ayd_alg_contra_prop}
	The category of $aYD$-contramodules for a Hopf algebroid $\mathcal{H}$ is  isomorphic to $\text{w-}\Zc_\hmodalg(\hmodalg^{op})$.
	
\end{proposition}

\begin{proof}
	
	Given a central element  $(M,\tau)\in\Zc_\hmodalg(\hmodalg^{op})$, define $\mu$ as in Lemma \ref{aYD-contramodule_alg lemma}. It satisfies aYD condition, so we just need to check that $M$ is a right $\mathcal{H}$-contramodule.
	
	By definition of the center and because all associativity maps are trivial, the diagram
	
	\begin{equation}\label{assoc}\xymatrix{W\la(V\la M)\ar[r]^{Id\la\tau}& W\la(M\ra V)\ar[r]& (W\la M)\ra V\ar[r]^{\tau\ra Id}  &(M\ra W)\ra V\ar[d]\\
		(W\ot V)\la M\ar[u]\ar[rrr]^{\tau} & & & M\ra(W\ot V) \\
	}\end{equation}
	commutes for any $V, W \in \hmodalg$.
Taking $f$ along the second row to $w\ot v\mapsto\mu(h\mapsto f(h_1 w\ot h_2 v))$ is the same as taking it the long way around, which is $w\ot v\mapsto\mu(x\mapsto\mu(h\mapsto f(xw\ot hv)))$. In particular, if we take $\mathcal{H}$ for  $V$ and $W$ and assume $v=w=1$, we will obtain the contraaction condition.  Next, we have that $R_l\la M\to M\ra R_l$ is the identity.

 Define: $m(r)= t_l(r)  m$, then
 \begin{equation}\label{unit}
 \mu(h\mapsto m(\epsilon_l(h s_l(r)))) = \mu(h\mapsto m(h \cdot r)) = \tau m(r)= m(r) =  t_l(r) m.\end{equation}
 The first equality is just an $\mathcal{H}$-module  structure of $R_l$.
	
	Finally, suppose that $f:M\to M'$ is a map in the center, then we have $$\xymatrix{H\la M\ar[d]_{Id\la f}\ar[r]^{\tau}& M\ra H\ar[d]_{f\ra Id}\ar[r]^{ev_1}& M\ar[d]_f\\
		H\la M'\ar[r]^\tau & M'\ra H\ar[r]^{ev_1} & M\\
	}$$ where the left square commutes by the centrality of $f$ and the right square commutes obviously.  Since the top row is $\mu_M$ and the bottom row is $\mu_{M'}$ so $f$ is a map of $aYD$-contramodules. \\
	
	Conversely, consider an $aYD$-contramodule $M$. By Lemma \ref{aYD-contramodule_alg lemma} it provides us with  natural morphisms $\tau_V: V \la M \rightarrow M \ra V$ for every $V \in \hmodalg$.  Observe that the considerations \eqref{assoc} and \eqref{unit} are of the if and only if kind, so  $\mu$ being a contraaction implies that the diagram \eqref{assoc} commutes and $\tau_{R_l} = id$.
	
	 If $\phi:M\to M'$ is a map of $aYD$-contramodules then $$(\phi \circ\tau f)(v)= \phi(\mu(h\mapsto f(hv)))=\mu(h\mapsto \phi(f(hv)))=\tau(\phi \circ f)(v)$$ so that $\phi$ is a morphism in $\Zc_\hmod(\hmod^{op})$. 	What has been shown so far is that if $M$ is an $aYD$-contramodule, then $(M,\tau)\in\Zc_\hmod(\hmod^{op})$ and any $g:M\to M'$ a morphism of $aYD$-contramodules induces a morphism between the corresponding central elements.

	As we noted in Lemma \ref{aYD-contramodule_alg lemma}, the constructions of central elements from $aYD$-contramodules and vice versa are inverses of each other, thus we have proved the isomorphism of categories.
\end{proof}

\begin{remark}
	Similar to Remark \ref{weak_center_quasi} we cannot prove nor disprove whether the center coincides with the weak center in the case of Hopf algebroids.
\end{remark}

Recall the following:
\begin{definition}
	An $aYD$-contramodule $M$ is called stable ($saYD$), if
	$\mu (r_m) = m$, where the map $r_m: \mathcal{H} \rightarrow M$, is given by $\ h\mapsto hm$.
\end{definition}

By Lemma \ref{weakstrong}, $\Zc'_\hmodalg(\hmodalg^{op})$ is the full subcategory of the center that consists of objects such that the identity map $Id\in\Hom_\mathcal{H}(M,M)$ is mapped to same via \begin{equation}\label{stabilitycond_alg}\Hom_\mathcal{H}(M,M)\simeq\Hom_\mathcal{H}(\id,M\la M)\simeq \Hom_\mathcal{H}(\id,M\ra M)\simeq \Hom_\mathcal{H}(M,M).\end{equation}  We immediately obtain the following Corollary of Proposition \ref{ayd_alg_contra_prop}.

\begin{corollary}\label{saydcor_alg}
	The category of $saYD$-contramodules for a Hopf algebroid $\mathcal{H}$ is  isomorphic to $\Zc'_\hmodalg(\hmodalg^{op})$.
\end{corollary}
\begin{proof}
	Let $M$ be an $aYD$-contramodule, then to ensure that the condition \eqref{stabilitycond_alg} is satisfied, we need that $$\tau(Id)(m)=\mu(h\mapsto hm)=\mu(r_m)=m$$ which is exactly the $saYD$ condition.  On the other hand, if $M$ is a central element that satisfies \eqref{stabilitycond_alg}, then it also satisfies that the chain of isomorphisms \begin{equation}\label{fullstabcond}\Hom_\mathcal{H}(V,M)\simeq\Hom_\mathcal{H}(\id,V\la M)\simeq\Hom_\mathcal{H}(\id,M\ra V)\simeq \Hom_\mathcal{H}(V,M) \end{equation} is identity for any $V\in\hmodalg$.  Take $V=\mathcal{H}$ and note that $r_m\in\Hom_\mathcal{H}(\mathcal{H},M)$ so that \eqref{fullstabcond} implies that $\tau(r_m)=r_m$ and so $\mu(r_m)=ev_1\tau(r_m)=ev_1(r_m)=m$, i.e., $M$ is a stable $aYD$-contramodule.
\end{proof}

Using Lemma \ref{reptrace}, we can  construct Hopf-cyclic cohomology for Hopf algebroids with $saYD$-contramodule coefficients.


\end{document}